\documentclass[reqno,11pt]{amsart}

\usepackage{fullpage,setspace}
\usepackage{cite}
\usepackage{tfrupee}  %

\usepackage{mathrsfs} %
\usepackage[bottom]{footmisc}  %
\usepackage{enumerate}   %
\usepackage{url,cite,fancyheadings}
\usepackage{enumerate}
\usepackage{amsmath,amssymb,amscd,mathrsfs,latexsym}
		\usepackage{rsfso}  %
		\usepackage{mathtools}
\usepackage{mathdots}
\usepackage{tikz}

\usepackage{ifthen}     %

\newcommand\pubpri[2]{%
\ifthenelse{\equal{\version}{public}}%
{{#1}}%
{\ifthenelse{\equal{\finalized}{no}}{\marginpar{\scshape\small Pubpri Alert}{#2}}{#2}}{}}
\newcommand\pubprinoalert[2]{%
\ifthenelse{\equal{\version}{public}}%
{{#1}}%
{#2}}
\newcommand\ignore[1]{}
\usepackage{color}
\providecommand\wantcolor{yes}   %
\definecolor{backgroundyellow}{cmyk}{.2,.1,.8,.2}
\definecolor{backgroundblue}{rgb}{0,0,1}
\definecolor{backgroundred}{rgb}{1,0,0}
\definecolor{backgroundmagenta}{cmyk}{0,1,0,0}
\ifthenelse{\equal{\wantcolor}{no}}{\renewcommand\color[1]{}}{}

\usepackage{enumitem}

\newcommand\mysubsubsection[1]{%
		\subsubsection{\sffamily\upshape\mdseries #1}}

\newcommand\mysss{\mysubsubsection}

\usepackage{chngcntr}
\counterwithin{figure}{section}
\usepackage{url,hyperref}  %
\hypersetup{colorlinks=true,
linkcolor=blue,
filecolor=magenta,
urlcolor=cyan,}
\usepackage{graphicx,wrapfig}
\providecommand{\theoremnumbering}{document}

\ifthenelse{\equal{\theoremnumbering}{section}}{\newtheorem{annotation}{Annotation}[section]}%
{\ifthenelse{\equal{\theoremnumbering}{subsection}}{\newtheorem{annotation}{Annotation}[subsection]}{\newtheorem{annotation}{Annotation}}}
\newtheorem{theorem}[annotation]{%
		Theorem}
\newtheorem{lemma}[annotation]{%
		Lemma}
\newtheorem{definition}[annotation]{%
		Definition}
\newtheorem{corollary}[annotation]{%
		Corollary}
\newtheorem{proposition}[annotation]{%
		Proposition}

\newtheorem{example}[annotation]{%
		Example}
\newcommand\bexample{\begin{example}\begin{rm}}
\newcommand\eexample{\end{rm}\hfill$\Box$\end{example}}
\newtheorem{examplenobox}[annotation]{%
		Example}
\newcommand\bexamplenobox{\begin{examplenobox}\begin{rm}}
\newcommand\eexamplenobox{\end{rm}\end{examplenobox}}
\newtheorem{exercise}[annotation]{%
		Exercise}
\newcommand\bexercise{\noindent\begin{exercise}\begin{rm}}
\newcommand\eexercise{\end{rm}\end{exercise}}
\newtheorem{notation}[annotation]{%
		Notation}
\newcommand\bnotation{\begin{notation}\begin{rm}}
\newcommand\enotation{\end{rm}\end{notation}}

\newtheorem{remark}[annotation]{%
		Remark}
\newcommand\bremark{\begin{remark}%
\begin{upshape}}
\newcommand\eremark{\end{upshape}%
\end{remark}}
\newenvironment{remark*}{%
\par\noindent{\scshape 
  Remark: }\begin{rm}}{\hfill\end{rm}\newline} 
\newcommand\bremarkstar{\begin{remark*}}
\newcommand\eremarkstar{\end{remark*}}
\newcommand\bdefn{\begin{definition}%
\begin{upshape}}
\newcommand\edefn{\end{upshape}%
\end{definition}}
\newtheorem{caveat}[annotation]{%
		Caveat}
\newcommand\bcaveat{\begin{caveat}%
\begin{upshape}}
\newcommand\ecaveat{\end{upshape}%
\end{caveat}}
\newenvironment{caveatstar}{%
\par\noindent{\scshape\bfseries
  Caveat: }\begin{rm}}{\end{rm}\newline} 
\newcommand\bcaveatstar{\begin{caveatstar}}%
\newcommand\ecaveatstar{\end{caveatstar}}
\newenvironment{myproof}{%
\par\noindent{\scshape 
  Proof: }\begin{rm}}{\hfill$\Box$\end{rm}\newline} 
\newcommand\bmyproof{\begin{myproof}}
\newcommand\emyproof{\end{myproof}}
\newenvironment{myproofnobox}{%
\par\noindent{\scshape Proof: }\begin{rm}}{\end{rm}\hfill\newline}
\newcommand\bmyproofnobox{\begin{myproofnobox}}
\newcommand\emyproofnobox{\end{myproofnobox}}
\newenvironment{myproofof}[1]{%
\par\noindent{\scshape 
  Proof (of~#1): }\begin{rm}}{\hfill$\Box$\end{rm}\newline} 
\newcommand\bmyproofof{\begin{myproofof}}
\newcommand\emyproofof{\end{myproofof}}
\newenvironment{myproofofnobox}[1]{%
\par\noindent{\scshape 
  Proof (of~#1): }\begin{rm}}{\end{rm}\hfill\newline} 
\newcommand\bmyproofofnobox{\begin{myproofofnobox}}
\newcommand\emyproofofnobox{\end{myproofofnobox}}
\newenvironment{solution}{%
\par\noindent{\scshape Solution: }\begin{rm}}{\hfill$\Box$\end{rm}\newline}
\newenvironment{solutionnobox}{%
\par\noindent{\scshape Solution: }\begin{rm}}{\end{rm}}
\newcommand\bsolution{\begin{solution}\begin{rm}}
\newcommand\esolution{\end{rm}\end{solution}}
\newcommand\bsolutionnobox{\begin{solutionnobox}\begin{rm}}
\newcommand\esolutionnobox{\end{rm}\end{solutionnobox}}
\newcommand\bthm{\begin{theorem}}
\newcommand\ethm{\end{theorem}}
\newcommand\bcor{\begin{corollary}}
\newcommand\ecor{\end{corollary}}
\newcommand\blemma{\begin{lemma}}
\newcommand\elemma{\end{lemma}}
\newcommand\bprop{\begin{proposition}}
\newcommand\eprop{\end{proposition}}
\newcommand\beqn{\begin{equation}}
\newcommand\eeqn{\end{equation}}
\newcommand\beqnstar{\begin{equation*}}
\newcommand\eeqnstar{\end{equation*}}
\numberwithin{equation}{section}
\usepackage{amsthm}
\usepackage{textcomp}
\usepackage{calrsfs,mathrsfs}  %
\newcommand\mtitle[1]%
{\marginpar{\sffamily\raggedright\upshape\small #1}}

\providecommand\finalized{yes}
\ifthenelse{\equal{\finalized}{yes}}{}{\marginparwidth=2cm}
\ifthenelse{\equal{\finalized}{yes}}%
{}%
{}
\ifthenelse{\equal{\finalized}{yes}}%
{\newcommand\checked[1]{}}%
{\newcommand\checked[1]{\marginpar{[{\ttfamily\upshape\tiny CHECKED: #1}]}}}
\ifthenelse{\equal{\finalized}{yes}}%
{\newcommand\spellchecked[1]{}}%
{\newcommand\spellchecked[1]{\marginpar{[{\ttfamily\upshape\tiny SPELLCHECKED: #1}]}}}
\providecommand\version{public}   %
\ifthenelse{\equal{\version}{public}}%
{\newcommand\mcomment[1]{}}%
{\newcommand\mcomment[1]{\marginpar{{\raggedright\sffamily\upshape\small
\begin{spacing}{0.75} #1\end{spacing}}}}}
\ifthenelse{\equal{\version}{public}}%
{\newcommand\fcomment[1]{}}%
{\newcommand\fcomment[1]{\footnote{#1}}}
\ifthenelse{\equal{\version}{public}}%
{\newcommand\comment[1]{}}%
{\newcommand\comment[1]{{\small #1}}}

\newcommand{\be}{\begin{enumerate}}
\newcommand{\ee}{\end{enumerate}}
\newcommand{\bi}{\begin{itemize}}
\newcommand{\ei}{\end{itemize}}
\newcommand{\beq}{\begin{equation}}
\newcommand{\eeq}{\end{equation}}
\newcommand{\complex}{\mathbb{C}}

\newcommand{\rationals}{\mathbb{Q}}
\newcommand{\integers}{\mathbb{Z}}

\newcommand{\prinheis}{\mathfrak{s}}
\newcommand{\alphac}{\alpha^\vee}

\DeclareMathOperator{\End}{End}

\newcommand{\lus}[1][t]{m^{\lambda}_{\mu}(#1)}
\newcommand{\lusztig}[2]{m^{#1}_{#2}(t)} 
\newcommand{\tko}[1][]{K(#1; \,t)}

\newcommand{\kost}[1][\mu]{K_{\lambda #1}(t)}

\newcommand{\hl}[1][\lambda]{P_{#1}(t)}

\newcommand{\lamtil}{\widetilde{\lambda}}
\newcommand{\afatype}[1][]{A_{#1}^{(1)}}
\newcommand{\afsl}[1][2]{\ensuremath{\widehat{\mathfrak{sl}}_{\,#1}}}

\newcommand{\scrb}{\mathcal{B}}
\newcommand{\vacm}{L(\Lambda_0)}
\newcommand{\vac}[1][0]{| #1 \rangle}

\newcommand{\fock}[1][]{\mathcal{F}_{#1}}
\newcommand{\focktil}{\widetilde{\fock}}

\newcommand{\coxn}{\mathbf{h}}
\newcommand{\coxdual}{\mathbf{h^\vee}}

\newcommand{\omd}[1][d]{\omega^{\,#1}}

\newcommand{\scr}{\mathcal}

\newcommand{\liegf}{\overline{\lieg}}
\newcommand{\liehf}{\overline{\lieh}}
\newcommand{\lieg}{\mathfrak{g}}
\newcommand{\lieh}{\mathfrak{h}}
\newcommand{\walg}{\mathcal{W}}
\newcommand{\woneinf}{\mathcal{W}_{1+\infty}}
\newcommand{\zhu}{\mathfrak{Zh}(\walg)}
\newcommand{\centug}{{Z}(\mathrm{U}\liegf)}
\newcommand{\vermaw}{\mathbf{M}}
\newcommand{\vermahigh}{\mathbf{1}}
\newcommand{\bas}{\mathcal{B}}

\newcommand{\homheis}[1][]{\widehat{\lieh}_{#1}}
\newcommand{\hhplus}[1][]{\homheis[#1]^+}
\newcommand{\hhminus}[1][]{\homheis[#1]^-}

\newcommand{\form}[2]{(#1 \,|\, #2)}
\newcommand{\aform}[2]{\langle #1 \, , \, #2 \rangle}
\newcommand{\cttinv}{\complex\left[t,t^{-1}\right]}
\newcommand{\twistga}[1][Q]{\complex_\varepsilon[#1]}
\newcommand{\roots}[1][]{\Delta_{#1}}
\newcommand{\confvec}{\omega}
\newcommand{\vqgrad}[1][d]{V_Q^{\,[#1]}}
\newcommand{\wgrad}[1][d]{\walg^{\,[#1]}}
\newcommand{\vqtilgrad}[1][d]{V_{\Qtil}^{\,[#1]}}

\newcommand{\Qfin}{Q}
\newcommand{\Qtil}{\widetilde{Q}}

\newcommand{\hwvec}{v_{\Lambda_0}}
\newcommand{\lattva}[1][\Qfin]{V_{#1}}
\newcommand{\gln}[1][\ell+1]{\ensuremath{\mathfrak{gl}_{#1}}\xspace}
\newcommand{\sln}[1][\ell+1]{\ensuremath{\mathfrak{sl}_{#1}}\xspace}
\newcommand{\woi}{\walg_{1+\infty}}
\newcommand{\lzo}{L^1_0}

\newcommand{\cch}[1][\lambda]{\gamma_{{}_{#1}}}
\newcommand{\mtop}[1][M]{{#1}_{\mathrm{top}}}
\newcommand{\scrD}[1][D]{\mathcal{#1}}
\DeclareMathOperator{\Sym}{S}
\DeclareMathOperator{\Aut}{Aut}
\DeclareMathOperator{\Inn}{Inn}
\DeclareMathOperator{\Ind}{Ind}
\DeclareMathOperator{\Ad}{Ad}
\DeclareMathOperator{\ad}{ad}
\DeclareMathOperator{\id}{id}
\DeclareMathOperator{\gr}{gr}
\DeclareMathOperator{\ch}{ch}
\DeclareMathOperator{\tr}{tr}
\DeclareMathOperator{\Hilb}{{\mathbf H}}
\DeclareMathOperator{\Res}{Res}

\usepackage{xspace}

\begin{document}
\title[]{The Brylinski filtration for affine Kac-Moody algebras and representations of $\mathcal{W}$-algebras}
\author{Suresh Govindarajan}
\address{Department of Physics\\
Indian Institute of Technology Madras\\
Chennai, India.}
\email{suresh@physics.iitm.ac.in}
\author{Sachin S. Sharma}
\address{Department of Mathematics and Statistics\\
Indian Institute of Technology\\
Kanpur, India.}
\email{sachinsh@iitk.ac.in}
\author{Sankaran Viswanath}
\address{The Institute of Mathematical Sciences, HBNI, Chennai, India.}
\email{svis@imsc.res.in}
\begin{abstract}
We study the Brylinski filtration induced by a principal Heisenberg subalgebra of an affine Kac-Moody algebra $\lieg$, a notion first introduced by Slofstra. The associated graded space of this filtration on dominant weight spaces of integrable highest weight modules of $\lieg$ has Hilbert series coinciding with Lusztig's $t$-analog of weight multiplicities. For the level 1 vacuum module $\vacm$ of affine Kac-Moody algebras of type $A$, we show that the Brylinski filtration may be most naturally understood in terms of representations of the corresponding $\walg$-algebra. We show that the sum of dominant weight spaces of $\vacm$ in the principal vertex operator realization forms an irreducible Verma module of $\walg$ and that the Brylinski filtration is induced by the Poincar\'{e}-Birkhoff-Witt basis of this module. This explicitly determines the subspaces of the Brylinski filtration. Our basis may be viewed as the analog  of Feigin-Frenkel's basis of $\walg$ for the $\walg$-action on the principal rather than on the homogeneous realization of $\vacm$.
\end{abstract}
\thanks{Corresponding author: Sankaran Viswanath}
\maketitle
\section{Introduction}
\subsection{}
We assume throughout that the ground field is $\complex$. Let $\lieg$ denote a symmetrizable Kac-Moody algebra. Let $L(\lambda)$  be an
integrable highest weight representation of $\lieg$ and $\mu$ a dominant weight of $L(\lambda)$. Lusztig's $t$-analog of weight multiplicity is defined to be the polynomial \cite{lusztig}:
\begin{equation}\label{eq:lustanalog}
  \lus = \sum_{w \in W} \varepsilon(w) \,\tko[w(\lambda + \rho) - (\mu+\rho)],
\end{equation}
where $\varepsilon$ is the sign character of the Weyl group $W$ of $\lieg$, $\,\rho$ is the Weyl vector and $K$ is the $t$-analog of Kostant's  partition function defined by:
\[ \sum_{\beta \in Q_+} \tko[\beta]\,e^{\beta} = \prod_{\alpha \in \roots[+]} \left(1-te^{\alpha}\right)^{-m_\alpha}. \]
Here $\roots[+]$ is the set of positive roots, $Q_+ = \integers_{\geq 0}(\roots[+])$ and $m_\alpha$ is the multiplicity of the root $\alpha$. At $t=1$, $\lus$ reduces to the weight multiplicity $\dim L(\lambda)_\mu$.
\subsection{}
When $\lieg$ is finite-dimensional, it is a classical fact that $\lus$ has non-negative integral coefficients. A principal nilpotent element of $\lieg$ induces a filtration on $L(\lambda)$, called the Brylinski-Kostant filtration. The associated graded space of its restriction to $L(\lambda)_\mu$ has Hilbert-Poincar\'{e} series $\lus$ \cite{rkbryl, broer, JLZ}. 

\subsection{}
For $\lieg$ of affine type, it was conjectured by Braverman-Finkelberg \cite{BF} that the analogous result holds, i.e., that the associated graded space of the principal nilpotent filtration on  $L(\lambda)_\mu$ has Hilbert-Poincar\'{e} series $\lus$. Slofstra \cite{slofstra} showed that this was false, but that it becomes true if the principal nilpotent filtration is replaced by the principal Heisenberg filtration \cite[Theorem 2.2]{slofstra}.  The latter is the filtration of the weight spaces induced by the positive part of the principal Heisenberg subalgebra  of $\lieg$ (see \S\ref{sec:bryl-def} below); we shall call this the Brylinski filtration, following Slofstra.

\subsection{} There is a second interpretation of the polynomials $\lus$.  In the finite-dimensional case, they coincide with the Kostka-Foulkes polynomials $\kost$, the coefficients that occur when the character of $L(\lambda)$ is expressed in the basis of the Hall-Littlewood polynomials $\hl[\mu]$ \cite{kato,rkg}. This result was generalized to all symmetrizable Kac-Moody algebras in \cite[Theorem 1]{svis-kfp}. In particular, when $\lieg$ is a simply-laced affine Kac-Moody algebra, this interpretation enables a closed-form computation of the $\lus$ for the level 1 vacuum module (basic representation) $\vacm$ of $\lieg$. This takes the generating function form \cite[Corollary 2]{svis-kfp}:
\begin{equation}\label{eq:tsf-level1}
\sum_{n \geq 0} \lusztig{\Lambda_0}{\Lambda_0 - n\delta} \, q^n = \prod_{k=1}^{\ell}{\prod_{n=1}^{\infty}{(1-t^{d_{k}}q^n)^{-1}}}, 
\end{equation}
where $d_{k}\,(1\leq k\leq \ell)$ are the degrees of the underlying finite-dimensional
simple Lie algebra $\liegf,\;\Lambda_0$ is the dominant weight of $\lieg$ of level 1 which vanishes on the Cartan subalgebra $\liehf$ of $\liegf$ and $\delta$ is the null root of $\lieg$. As shown in \cite{svis-kfp}, equation~\eqref{eq:tsf-level1} is essentially equivalent to 
the $q$-Macdonald-Mehta constant term identity \cite[Theorem 5.3]{dmmc} proved by Cherednik via Double Affine Hecke Algebra (DAHA) techniques. We recall here that the degrees $d_k$ are nothing but the exponents of $\liegf$ plus one.

\subsection{}
Consider the subspace $Z=\vacm^{\liehf}$ of $\liehf$-invariants of $\vacm$. It admits the (``energy'') grading:
\[ Z= \bigoplus_{n \geq 0} Z_n := \bigoplus_{n \geq 0} L(\Lambda_0)_{\Lambda_0 - n \delta}. \]
We restrict the Brylinski filtration to $Z$; the associated graded space $\gr Z$ is then a bi-graded vector space. From Slofstra's result mentioned above \cite[Theorem 2.2]{slofstra} and equation~\eqref{eq:tsf-level1}, we obtain its two-variable Hilbert-Poincar\'{e} series:
\begin{equation} \label{eq:hilbZ}
  \Hilb(\gr Z; \,t,q) := \sum_{n \geq 0} \sum_{i \geq 0} \dim \left( F^i Z_n \,/ F^{i-1}  Z_n\right) \, t^i q^n = \prod_{k=1}^{\ell} \prod_{n=1}^{\infty}{(1-t^{d_{k}}q^n)^{-1}},
\end{equation}
where $F^i \vacm$ denotes the $i^{\,\mathrm{th}}$ subspace in the Brylinski filtration (\S\ref{sec:bryl-def}) of $\vacm$ and $F^i U = U \cap F^i \vacm$ for a subspace $U$ of $\vacm$.
We note that the $t$-degree arises from the Brylinski filtration while the $q$-degree is the energy grading. 

\subsection{}

Given a subspace $U$ of $\vacm$, a basis $\scrb$ of $U$ is said to be {\em Brylinski-compatible} if for all $d \geq 0$, $\scrb \cap F^d U$ is a basis of $F^d U$; in other words if the image of $\scrb$ in $\gr U$ is a basis of $\gr U$. The main result of this paper is the construction of a natural Brylinski-compatible basis of $Z$, for the untwisted affine Lie algebras of type $A$.

To state our main theorem, we briefly recall the relevant facts about $\walg$-algebras in general. Let $\lieg$ denote a simply-laced affine Lie algebra, $\liegf$ the underlying finite-dimensional simple Lie algebra and $\walg = \walg(\liegf)$ the corresponding $\walg$-algebra. This is a vertex algebra which can be realized as a subalgebra of the lattice vertex algebra $\lattva$, where $\Qfin$ is the root lattice of $\liegf$. Let $d_1 \leqslant d_2 \leqslant \cdots \leqslant d_\ell$ be the list of degrees of $\liegf$ as in \eqref{eq:tsf-level1}.
The following theorem, due to Feigin and Frenkel \cite{feigin-frenkel}, states an important fact about $\walg$-algebras.
\begin{theorem}[Feigin-Frenkel]\label{thm:ff} 
There exist elements $\omd[i] \in \walg$ of degree $d_i$ such that $\walg$ is freely generated by  $\omd[1],\, \omd[2],\, \cdots,\, \omd[\ell]$.
\end{theorem}

\noindent
For $\lieg = \afatype[\ell] \; (\ell \geqslant 1)$, the degrees $d_i$ are $2, 3, \cdots, \ell+1$. The field $Y(\omd[i],z)$ has conformal weight $d_i$ and we write:
\[ Y(\omd[i], z) = \sum_{n \in \integers} \omd[i]_n \,z^{-n - d_i}.\]
Here $\omd[1]$ is the conformal vector of $\walg$, which corresponds to the Virasoro field. Let $\vac$ denote the vacuum vector of $\walg$. Theorem~\ref{thm:ff} states that the following is a basis of $\walg$ (in PBW fashion):

\beq\label{eq:ffbasis}
\omd[p_1]_{k_1}\; \omd[p_2]_{k_2}\; \cdots\; \omd[p_r]_{k_r} \vac
\eeq
\smallskip
with (i) $r \geqslant 0$\quad(ii) $\ell \geqslant p_1 \geqslant p_2 \geqslant \cdots \geqslant p_r \geqslant 1$ \quad(iii) $k_j \leqslant -d_{p_j}$ for all $j$\quad   and (iv) if $p_i = p_{i+1}$, then $k_i \leqslant k_{i+1}$.

\subsection{}\label{sec:prin-w-mod}

The level 1 vacuum module $\vacm$ of $\lieg$ has many realizations in terms of vertex operators, one for each conjugacy class of the Weyl group $W(\liegf)$ of $\liegf$ \cite{kac-peterson-112constructions}. The principal realization of $\vacm$  corresponds to the conjugacy class of a Coxeter element $\sigma \in W(\liegf)$. This realization endows $\vacm$ with the structure of a $\sigma$-twisted representation of the lattice vertex algebra $\lattva$. When restricted to the vertex subalgebra $\walg \subset \lattva$, this representation becomes untwisted, since $\walg$ is pointwise fixed by $\sigma$ (and indeed by every Weyl group element) \cite{bakalov-milanov}. 
The following is our main theorem.
\begin{theorem}\label{thm:main-thm-intro}
Let $\lieg = \afatype[\ell]$. Consider the principal realization of $\vacm$ as a $\walg$-module and let $\hwvec$ denote a highest weight vector of $\vacm$.
\be
\item[(a)] The $\walg$-submodule of $\vacm$ generated by $\hwvec$ is exactly the space $Z = \vacm^{\liehf}$.
\item[(b)] $Z$ is isomorphic to a Verma module of $\walg$, and is irreducible.
\item[(c)] The following vectors form a Brylinski-compatible basis of $Z$:
\beq\label{eq:ffbasis-twisted}
\omd[p_1]_{k_1}(\sigma)\; \omd[p_2]_{k_2}(\sigma)\; \cdots\; \omd[p_r]_{k_r}(\sigma) \,\hwvec,
\eeq
\smallskip
where (i) $r \geqslant 0$\quad(ii) $\ell \geqslant p_1 \geqslant p_2 \geqslant \cdots \geqslant p_r \geqslant 1$ \quad(iii) $k_j \leqslant -1$ for all $j$\quad   and (iv) if $p_i = p_{i+1}\,$, then $k_i \leqslant k_{i+1}$.
\ee
\end{theorem}

We note that the modes $\omd[i]_n(\sigma)$ in \eqref{eq:ffbasis-twisted} now refer to the action of $\omd[i] \in \walg$ on the principal realization of $\vacm$:
\begin{equation*} \label{eq:modesdef}
  Y_{{}_{\vacm}}(\omd[i], z) = \sum_{n \in \integers} \omd[i]_n(\sigma) \,z^{-n - d_i}.
\end{equation*}

Unlike the Feigin-Frenkel basis  \eqref{eq:ffbasis} which involves only the modes $k_j \leq -d_{p_j}$, here we have $k_j \leq -1$ for all $j$.
We may view the Feigin-Frenkel basis as the analog of \eqref{eq:ffbasis-twisted} for the homogeneous realization of $\vacm$, with $\liehf$-invariants replaced by $\liegf$-invariants (see Remark~\ref{rem:walg-ginv}).
As a corollary to Theorem~\ref{thm:main-thm-intro}, we obtain an explicit description of the subspaces of the Brylinski filtration:
\begin{corollary}\label{cor:bryl-subspaces} Let $n,d \geq 0$. 
 The subspace $F^d Z_n$ has a basis given by the vectors in \eqref{eq:ffbasis-twisted} satisfying $\sum_{i=1}^r d_{p_i} \leq d$ and $\sum_{i=1}^r k_i = -n$.
\end{corollary}

\medskip
Verma modules of $\walg$-algebras are generically irreducible \cite{arakawa}. We show that the highest weight of the $\walg$-module $Z$ is a generic weight (see Proposition~\ref{prop:primchar}) indirectly, using results of Frenkel-Kac-Radul-Wang \cite{fkrw}  which relate representations of $\walg$-algebras in type $A$ to those of  $\widehat{\mathcal{D}}$, the universal central extension of the Lie algebra of regular differential operators on the circle. We however expect Theorem~\ref{thm:main-thm-intro} to hold for all simply-laced affine algebras.

\medskip
We remark that the representation of $\walg$ obtained by restricting the principal realization of $\vacm$  has been previously considered by Bakalov-Milanov \cite{bakalov-milanov} and other authors (see references in \cite{bakalov-milanov}) for its role in the study of singularities. Representations of $\walg$-algebras obtained by restricting more general twisted representations of affine Kac-Moody algebras have also been recently studied via heuristic character formulas in \cite{BGM-physics}. However, to the best of our knowledge, a rigorous proof of Theorem~\ref{thm:main-thm-intro}(b) does not appear in the literature. Further, the connection to Cherednik's $q$-Macdonald-Mehta constant term identity in type $A$ provided by Theorem~\ref{thm:main-thm-intro}(c) does not seem to have been noticed before.

\smallskip
The paper is organized as follows. Sections~\ref{sec:vas}-\ref{sec:twisted-mods} contain background and establish notation on lattice vertex algebras and their twisted modules. A reader acquainted with these notions can skip directly to section~\ref{sec:bryl-fil}, which describes the Brylinski filtration on $\vacm$ in the principal realization. Section~\ref{sec:walg} recalls the main facts about $\walg$-algebras and their Verma modules. Sections~\ref{sec:zmzeta}-\ref{sec:fkrw} culminate in proving that $Z$ is an irreducible Verma module of $\walg$. Finally, section~\ref{sec:bryl-comp} completes the proof of Theorem~\ref{thm:main-thm-intro} by showing that the natural PBW basis of this Verma module is in fact compatible with the Brylinski filtration.

\subsection{Acknowledgments} It is a pleasure to thank Bojko Bakalov for useful discussions and pointers to relevant literature at an early stage of this work. The authors also thank an anonymous referee whose comments helped improve the exposition. Sachin Sharma also thanks the Institute of Mathematical Sciences, where parts of this work were done, for its excellent hospitality.

\subsection{Funding:} Partial financial support was received from the Department of Science and Technology, Government of India through grants EMR/2016/001997 (SG) and MTR/2019/000071 (SV), from the Department of Atomic Energy, Government of India under a XII plan project (SV) and from the Indian Institute of Technology Kanpur through a faculty initiation grant (SS).

\section{Vertex algebras, representations}\label{sec:vas}
We first recall the primary notions about vertex algebras and fix notation \cite{kac-vafb,bakalov-milanov,bakalov-kac}.
\subsection{Vertex algebras}
A vertex algebra is a vector space $V$, with a vector  $\vac \in V$ (vacuum vector) and a map $Y: V \to \End V[[z,z^{-1}]]$ (state-field correspondence):
\[ Y(a,z) = \sum_{n \in \integers} a_{(n)} \, z^{-n-1} \]
satisfying the following axioms \cite[Prop 4.8(b)]{kac-vafb}: (a) $Y(a,z)\,b$ is a Laurent series in $z$ for all $a,b \in V$ (b) $Y(\vac,z) = \id_V$ (c) $a_{(-1)}\vac = a$ for all $a \in V$ and (d) the modes satisfy the Borcherds identity:
\begin{equation} \label{eq:borcherds-id} 
  \begin{split}
    \sum_{j=0}^\infty (-1)^j \binom{n}{j} \Big( a_{(m+n-j)} \left(b_{(k+j)} \,c\right) - (-1)^n\, b_{(k+n-j)} \left(a_{(m+j)} \,c\right) \Big)
    \\= \sum_{j=0}^\infty \binom{m}{j} \left( a_{(n+j)}\,b \right)_{(k+m-j)} \,c
\end{split}
  \end{equation}
for all $a,b,c \in V$ and $m,n,k \in \integers$.

\subsection{Strongly generating subset}
We recall that a vertex algebra $V$ is said to be strongly generated by a subset $X \subset V$ if $V$ is spanned by the vectors obtained by repeated actions of the negative modes of elements of $X$ on the vacuum $\vac$. In other words:

\[ V = \text{ span } \{ x^1_{(k_1)} \, x^2_{(k_2)} \,\cdots\, x^r_{(k_r)} \,\vac: \; r \geqslant 0, \,x^i \in X, \,k_i <0\}.\]

\subsection{The Heisenberg vertex algebra}\label{sec:heisva}
Let $\liehf$ be a finite-dimensional vector space with a symmetric nondegenerate bilinear form $\form{\cdot}{\cdot}$. 
We denote the affinization\footnote{The $t$ which occurs here is not to be confused with the indeterminate $t$ of \eqref{eq:lustanalog}. Since the two don't appear together in this paper, we permit ourselves this mild notational conflict so as to adhere to standard notations as much as possible.} of $\liehf$ by
$ \homheis = \liehf \otimes_{\complex} \cttinv  \oplus \complex K, $
where $K$ is the central element. Letting $ht^j$ denote the element $h \otimes t^j$, the Lie bracket is defined by $[h_1 t^j, h_2 t^k] = \delta_{j+k,0} \, \form{h_1}{h_2} \,K$, for $h_1,h_2 \in \liehf$ and $j,k \in \integers$. Thus $\homheis$ is isomorphic to a {\em Heisenberg Lie algebra}.
Let
\[ \fock = \Ind_{\,\hhplus \oplus \complex K}^{\homheis} \complex = U\homheis \otimes_{U(\hhplus \oplus \complex K)} \complex\]
be the highest weight irreducible representation of $\homheis$, where $\hhplus:=\liehf \otimes \complex[t]$ acts trivially on $\complex$ and $K$ acts as $1$. This is called the Fock space of $\homheis$. 
We may identify $\fock$ (as a vector space) with $\Sym \hhminus$, where $\hhminus := \liehf \otimes t^{-1}\complex[t^{-1}]$.
We have a grading:
$  \fock = \bigoplus_{p \geq 0} \fock^{[p]},$
where $ \fock^{[p]}$ is spanned by $\left(\liehf \otimes t^{j_1}\right) \cdots \left(\liehf \otimes t^{j_r}\right)$ with $j_i <0, \; \sum j_i = -p$.
The space $\fock$ forms a vertex algebra, called the Heisenberg vertex algebra. It is strongly generated by $\liehf \otimes t^{-1}$, with state-field correspondence determined by:
\begin{equation}\label{eq:heisva-statefield}
Y(ht^{-1},z) = \sum_{n \in \integers} (h t^n) \, z^{-n-1} \text{ for } h \in \liehf.
\end{equation}
This vertex algebra has a conformal vector \cite[Prop. 3.5]{kac-vafb}:
\begin{equation}\label{eq:confvec-heisva}
  \confvec = \frac{1}{2} \sum_{i=1}^{\dim \liehf} a^i_{(-1)} \, b^i_{(-1)} \,\vac,
\end{equation}
where  $\{a^i\}$ and $\{b^i\}$ are bases of $\liehf$ such that $\form{a^i}{b^{\,j}} = \delta_{ij}$ and the vacuum vector $\vac$ is the highest weight vector $1 \otimes 1$ of $\fock$. The central charge of the Virasoro field $Y(\omega,z) = \sum_{n \in \integers} \omega_n \,z^{-n-2}$ is $\dim \liehf$.

\subsection{Representations}
A representation of the vertex algebra $V$ (or $V$-module) is a vector space $M$ together with fields $Y_M(a,z) = \sum_{n \in \integers} a_{(n)} \, z^{-n-1}$ with $a_{(n)} \in \End M$ for each $a \in V, n \in \integers$. The axioms are (i) $Y_M(\vac,z) = \id_M$, (ii) $Y_M(a,z) \, m$ is a Laurent series in $z$ for each $a \in V, \,m \in M$, and (iii) the Borcherds identity \eqref{eq:borcherds-id} holds for all $a,b \in V$, $c \in M$ and all integers $m,n,k$.

\subsection{The oscillator representations of $\fock$} \label{sec:oscrep}
For $\lambda \in \liehf^*$, let $\complex_\lambda$ be the one-dimensional $(\hhplus \oplus \complex K)$-module on which $\liehf \otimes t\complex[t]$ acts trivially,  $h \otimes 1$ acts as $\lambda(h)$ for all $h \in \liehf$ and $K$ acts as $1$. Let 
\[ \pi_\lambda = \Ind_{\,\hhplus \oplus \complex K}^{\homheis} \complex_\lambda. \]
This is a highest weight irreducible representation of $\homheis$, called the oscillator representation. We let $\vac[\lambda]$ denote the highest weight vector $1 \otimes 1$ of $\pi_\lambda$. The space $\pi_\lambda$ is also a representation of the vertex algebra $\fock$, with the fields given by the same formula as in \eqref{eq:heisva-statefield}, but now with the $ht^n$ interpreted as operators on $\pi_\lambda$.
The conformal vector $\omega \in \fock$ induces a grading on $\pi_\lambda$; more precisely, we have $\pi_\lambda = \bigoplus_{k \in \complex} \, \pi_\lambda^{[k]}$, where $\pi_\lambda^{[k]} = \{x \in \pi_\lambda: \omega_0 (x) = kx\}$. For instance, one has $\omega_0 \vac[\lambda] = (|\lambda|^2/2) \vac[\lambda]$. The {\em character} of $\pi_\lambda$ is defined to be $\tr_{\pi_\lambda} q^{\omega_0}$ and is given by \cite[(3.4)]{kac-raina}:
\[ \tr_{\pi_\lambda}(q^{\omd[]_0}) = \frac{q^{|\lambda|^2/2}}{\varphi(q)^{\dim \liehf}}, \]
where $\varphi(q) = \prod_{n \geq 1} (1-q^n)^{-1}$ is Euler's $\varphi$-function.

\subsection{Twisted representations}

Let $\sigma$ be a finite order automorphism of a vertex algebra $V$, of order $d$ say. Then $\sigma$ is diagonalizable, with eigenvalues $\eta^j$, where $\eta = e^{\,2\pi i/d}$ and $j \in \integers$.
A $\sigma$-twisted representation of $V$ is a vector space $M$ together with fields $Y_M(a,z) \in \End M[[z^{1/d}, \, z^{-1/d}]]$. Given $a \in V$ such that $\sigma(a) = \eta^{-j} \,a$, we require:
\[ Y_M(a,z) = \sum_{n \in \frac{j}{d} + \integers} a_{(n)} \, z^{-n-1} \] with $a_{(n)} \in \End M$, i.e., the modes for $n \not\in \frac{j}{d} + \integers$ are zero. This is equivalent to $Y_M(\sigma a, z) = Y_M(a,e^{2\pi i}z)$ for all $a \in V$. 
The axioms are now \cite[Remark 3.1]{bakalov-kac}, \cite[\S 3.1]{bakalov-milanov}: (i) $Y_M(\vac,z) = \id_M$ (ii) $Y_M(a,z) \, m$ is a Laurent series in $z^{1/d}$ for each $a \in V, \, m \in M$ and (iii) the Borcherds identity \eqref{eq:borcherds-id} holds for all $a,b \in V$, $c \in M$ and for all $n \in \integers$, $m,k \in \frac{1}{d} \integers$.

\subsection{The twisted Heisenberg Fock space}
Let $\liehf$ be a finite-dimensional vector space with a symmetric nondegenerate bilinear form $\form{\cdot}{\cdot}$ and 
let $\varphi$ be a linear automorphism of $\liehf$ of order $m$, preserving $\form{\cdot}{\cdot}$. This induces an automorphism (also denoted $\varphi$) of the  vertex algebra $\fock$ of \S\ref{sec:heisva}.
One can extend $\varphi$ to an operator on the space $\liehf \otimes \complex[t^{1/m},t^{-1/m}] \oplus \complex K$ via:
\[
\varphi(h \otimes t^p) = \varphi(h) \otimes e^{2\pi i p} \,t^p, \;\;\;\; \varphi(K) = K \;\;\;\;\text{ for all } h \in \liehf, \; p \in \frac{1}{m} \integers.
\]
The $\varphi$-twisted Heisenberg algebra $\homheis[\varphi]$ is defined to be the set of fixed points of $\varphi$  \cite[\S 3.3]{bakalov-milanov}. It has a basis comprising the elements $K$ and $ht^p:=h \otimes t^p$, where $h \in \liehf, \, mp \in \integers$ such that $\varphi(h) = e^{-2\pi i p} h$ and the Lie bracket is:
$[a  t^p, b  t^q] = p \,\delta_{p,-q} \form{a}{b} K $
with $K$  central.
Let $\hhplus[\varphi]$ (respectively $\hhminus[\varphi]$) denote the span of the elements $h t^p$ of the above form for $p \geq 0$  (respectively $p<0$). As in \S\ref{sec:heisva}, we define the $\varphi$-twisted Fock space
\begin{equation} \label{eq:sigtwistfock}
  \fock[\varphi] = \Ind_{\,\hhplus[\varphi] \oplus \complex K}^{\,\homheis[\varphi]} \complex,
\end{equation}
where $\hhplus[\varphi]$ acts trivially on $\complex$ and $K$ acts as $1$.
Then $\fock[\varphi] \cong \Sym \hhminus[\varphi]$ (as vector spaces) becomes a $\varphi$-twisted representation of $\fock$ \cite[\S 3.3]{bakalov-milanov}. For $h \in \liehf$ with $\varphi(h) = e^{-2\pi i p} h$, we have
\begin{equation} \label{eq:twistatefield}
Y_{\fock[\varphi]}(ht^{-1},z) = \sum_{n \in p + \integers} (h t^n) \, z^{-n-1}.
\end{equation}

\subsection{Product identity}
For later use, we record the following ``product identity'' of Bakalov-Milanov \cite[Proposition 3.2]{bakalov-milanov} for twisted representations, rewritten in terms of modes:

\begin{proposition} (Bakalov-Milanov)
  Let $V$ be a vertex algebra, $\sigma$ an automorphism of $V$ of finite order $d$ and $M$ a $\sigma$-twisted $V$-module. Let $a, b \in V$ and $N_{ab}$ be a non-negative integer such that $a_{(k)}b =0$ for all $k \geq N_{ab}$. Fix $r \in \frac{1}{d}\integers$,  $n \in \integers$. Then for any $c \in M$, we have:
  \begin{equation}\label{eq:prodidentity}
  \left(a_{(n)}b\right)_{(r)} \, c = \displaystyle\sum_{\substack{p,q \in \frac{1}{d}\integers \\ p+q = r+n}} \kappa({p}) \; a_{(p)} b_{(q)} c,
    \end{equation}
    where $\kappa({p})$ is a scalar, given by $\kappa({p}) = \displaystyle\sum_{m=0}^N (-1)^{N-m} \binom{N}{m}\binom{m-p-1}{N-n-1}$ with  $N = \max(N_{ab}, n+1)$.
\end{proposition}

\section{The root lattice vertex algebra and its automorphisms}\label{sec:rlva}

\subsection{The lattice vertex algebra $\lattva[Q]$} \label{sec:rl-va}
We follow \cite{kac-vafb,bakalov-milanov,dong-nagatomo}. Let $\liegf$ denote a finite-dimensional simple Lie algebra over $\complex$. We further assume $\liegf$ is simply-laced, i.e., of types $A, D$ or $E$. Let $\roots$ denote the set of roots, $Q = \integers \roots$ denote the root lattice and $\liehf = \complex \otimes_\integers Q$ the Cartan subalgebra of $\liegf$. We assume $\form{\cdot}{\cdot}$ is the Killing form normalized such that $\form{\alpha}{\alpha} = 2$ for all roots $\alpha$.
As in \S\ref{sec:heisva}, we consider the Heisenberg Lie algebra
$ \homheis = \liehf[t,t^{-1}]  \oplus \complex K $
and let $\fock$ denote its Fock space. 

Let $\twistga$ denote the twisted group algebra of $Q$ with multiplication $e^\alpha e^\beta = \varepsilon(\alpha,\beta)\, e^{\alpha+\beta}$, where $\varepsilon: Q \times Q \to \{\pm 1\}$ is a bimultiplicative cocycle satisfying $\varepsilon(\alpha,\beta)\varepsilon(\beta,\alpha) = (-1)^{\form{\alpha}{\beta}}$.
The lattice vertex algebra $\lattva[Q]$ is defined to be the space
\[ \lattva[Q] = \fock \otimes_\complex \twistga. \]
This can be made into an $\homheis$-module by declaring each $1 \otimes e^\beta, \; \beta \in Q$ to be a highest weight vector, of highest weight $\beta$, i.e., $ht^k \, (1 \otimes e^\beta) =  \delta_{k,0} \, \aform{h}{\beta}\,(1 \otimes e^\beta)$ for $k \geqslant 0, \, h \in \liehf$. We let $h_n$ denote the operator $ht^n\otimes \mathrm{id}$ on $\fock \otimes_\complex \twistga$. We have the state-field correspondences:
\begin{align}
Y(ht^{-1} \otimes 1,z) &= \sum_{n \in \integers} h_n \, z^{-n-1}, \label{eq:heis-action-lattva}\\
Y(1 \otimes e^\alpha,z) &=  e^\alpha\,z^{\alpha_0}\,\exp \left( -\sum_{n <0} \frac{z^{-n}}{n} \,\alpha_n \right) \exp \left( -\sum_{n >0} \frac{z^{-n}}{n} \,\alpha_n \right), \label{eq:Q-action-lattva}
\end{align}
where  $h \in \liehf, \,\alpha \in Q$. Here, $e^\alpha$ denotes the left multiplication operator by $1 \otimes e^\alpha$, the operator $z^{\alpha_0}$ acts as $z^{\alpha_0} \,(\zeta \otimes e^\beta) = z^{\form{\alpha}{\beta}} \, (\zeta \otimes e^\beta)$ for $\zeta \in \fock$ and we identify $\alpha \in \liehf^*$ with its dual in $\liehf$ under the normalized Killing form. The vacuum vector is $\vac = 1 \otimes e^0$. The elements  $\{ht^{-1} \otimes 1: \,h \in \liehf\}$ and $\{1 \otimes e^\alpha: \,\alpha \in \roots\}$ strongly generate $\lattva[Q]$.
It is clear that $\fock$ is a vertex subalgebra of $\lattva$.

\subsection{Conformal vector, grading}

The conformal vector of $\lattva[Q]$ is given by:
\begin{equation}\label{eq:conflva}
  \confvec = \frac{1}{2} \sum_{i=1}^{\dim \liehf} a^i_{-1} \, b^i_{-1}\,\vac,
\end{equation}
where $\{a^i\}$ and $\{b^i\}$ are bases of $\liehf$ such that $\form{a^i}{b^{\,j}} = \delta_{ij}$.
Let the corresponding Virasoro field be $L(z) = Y(\confvec,z) = \sum_n L_n \,z^{-n-2}$ and define $\vqgrad := \{v \in \lattva[Q]: L_0 \,v = d\, v\}$.
This defines a vertex algebra grading $\lattva[Q] = \displaystyle\bigoplus_{d \in \integers_+} \,\vqgrad$, with $ht^{-n} \otimes 1 \in \vqgrad[n]$ and $1 \otimes e^\alpha \in \vqgrad[\form{\alpha}{\alpha}/2]$ for all $n \geqslant 0$ and $\alpha \in Q$. In particular, $\vqgrad[1]$ is spanned by the elements $ht^{-1} \otimes 1$ and $1 \otimes e^\alpha$ for $h \in \liehf, \alpha  \in \roots$. Hence $\vqgrad[1]$ strongly generates $\lattva[Q]$.

\subsection{Derivations and automorphisms}

We follow \cite{dong-nagatomo}. An automorphism of the lattice vertex algebra $\lattva[Q]$  is a vector space isomorphism $\varphi: \lattva[Q] \to \lattva[Q]$ that satisfies  $\varphi(\confvec) = \confvec$ and $\varphi ( a_{(n)}\,b ) = \varphi(a)_{(n)} \, \varphi(b)$ for all $a,b \in \lattva[Q], n \in \integers$.  A derivation $D: \lattva[Q] \to \lattva[Q]$ is a linear map that satisfies $D\confvec=0$ and $D ( a_{(n)}\,b ) = (Da)_{(n)} \, b + a_{(n)}\,Db$  for all $a,b \in \lattva[Q], n \in \integers$. If $D$ is a derivation, then $\exp D$ defines an automorphism of $\lattva[Q]$.
An automorphism fixes the vacuum vector $\vac$, while a derivation annihilates it. Further, derivations and automorphisms preserve the grading of $\lattva[Q]$, in particular they preserve the space $\vqgrad[1]$. Since this space strongly generates $\lattva[Q]$, a derivation or automorphism is uniquely determined by its action on $\vqgrad[1]$. 

\subsection{Inner automorphisms}\label{sec:innauts}
The space $\vqgrad[1]$ becomes a Lie algebra  under the bracket $[a,b] = a_{(0)}b$ for $a, b \in \vqgrad[1]$. It is isomorphic to $\liegf$, under an isomorphism mapping $(ht^{-1} \otimes 1)$ to $h$ and $(1 \otimes e^\alpha)$ to an element of the root space $\liegf_\alpha$ for each $h \in \liehf, \alpha \in \roots$. We identify $\vqgrad[1]$ with $\liegf$ via this isomorphism.

\smallskip
Given $X \in \vqgrad[1]$, the map $X_{(0)}$ is a derivation of $\lattva[Q]$. Restricted to $\vqgrad[1]$, this coincides with the adjoint action of $X$ on $\vqgrad[1]$ viewed as a Lie algebra. We let $\Inn \lattva[Q]$ denote the group of inner automorphisms of $\lattva$. This is the subgroup of $\Aut \lattva[Q]$ generated by $\{\exp X_{(0)}: X \in \vqgrad[1]\}$. Further, since $\vqgrad[1]$ strongly generates $\lattva[Q]$, it is clear that $\Inn \lattva[Q]$ is isomorphic to the group $\Inn \liegf = \left\langle \,\exp (\ad X): \,X \in \liegf \,\right\rangle$ of inner automorphisms of $\liegf$ \cite[\S 2]{dong-nagatomo}.

Given $\varphi \in \Inn \liegf$, say $\varphi = \exp (\ad X^1) \,\exp (\ad X^2) \ldots \exp (\ad X^p)$ with $X^i \in \liegf \cong \vqgrad[1] $, we denote by $\widetilde{\varphi} \in \Inn \lattva[Q]$ its unique lift to an automorphism of $\lattva[Q]$:
\begin{equation}\label{eq:lift_of_inner}
  \widetilde{\varphi} = \exp X_{(0)}^1 \,\exp X_{(0)}^2 \cdots \exp X_{(0)}^p.
\end{equation}
It is clear that if $\varphi$ has (finite) order $d$, then so does $\widetilde{\varphi}$.

\section{Twisted modules of the lattice vertex algebra as $\lieg$-modules}\label{sec:twisted-mods}

In this section, let $\lieg$ denote a simply-laced affine Lie algebra. Let $\liegf$ denote the underlying finite-dimensional simple Lie algebra with root lattice $Q$ and $V_Q$ the corresponding lattice vertex algebra.

\subsection{}\label{sec:twist1}
Let $m \geq 1$ and consider the Lie algebra $\liegf \otimes \complex[t^{1/m},t^{-1/m}] \oplus \complex K \oplus \complex d$. Here $K$ is central and the other Lie brackets are given by $[a \otimes t^p, b \otimes t^q] = [a,b]\otimes t^{p+q} + p \,\delta_{p,-q} \form{a}{b} K, \; [d, a\otimes t^p] = p (a \otimes t^p)$ for  $a, b \in \liegf$ and $p,q \in (1/m)\integers$.

Now let $\varphi$ be an inner automorphism of $\liegf$ of order $m$.  
Extend the action of $\varphi$ to the space $\liegf \otimes \complex[t^{1/m},t^{-1/m}] \oplus \complex K \oplus \complex d$ via:
\[
\varphi(x \otimes t^p) = \varphi(x) \otimes e^{2\pi i p} \,t^p, \;\;\;\; \varphi(K) = K, \;\;\;\;\varphi(d) = d\;\;\;\;\text{ for all } x \in \liegf, \; p \in \frac{1}{m} \integers.
\]
Define $\lieg[\varphi]$ to be the set of fixed points of $\varphi$. This is a Lie subalgebra and has a basis comprising the elements $K, d$ and $a \otimes t^p$, where $a \in \liegf, \, mp \in \integers$ such that $\varphi(a) = e^{-2\pi i p} a$. 
Since $\varphi$ is an inner automorphism of $\liegf$, the Lie algebra $\lieg[\varphi]$ is isomorphic to the affine Lie algebra $\lieg$ \cite[Theorem 8.5]{kac}.

\subsection{}\label{sec:twist2}
Let $\widetilde{\varphi}$ denote the unique lift of $\varphi$ to an automorphism of $\lattva[Q]$ as in \eqref{eq:lift_of_inner}.
Suppose $M$ is a $\widetilde{\varphi}$-twisted $\lattva[Q]$-module. Let $Y_M(a,z) = \sum_{p \in (1/m)\integers} a_{(p)} \, z^{-p-1}$ be the corresponding fields. Then $M$ becomes a $\lieg[\varphi]$-module via the action:
\[
a \otimes t^p \mapsto a_{(p)}, \;\; K \mapsto 1, \;\; d \mapsto -\omd[]_0,
\]
where $p \in \frac{1}{m}\integers,\, \varphi(a) = e^{-2\pi i p} a$ and $\omd[]$ is the conformal vector of $\lattva$ \cite{lepowsky-calctwistvo}.
    
\subsection{The automorphisms $\sigma, \zeta$ and $\psi$}\label{sec:twist3}
In what follows, we consider two specific automorphisms of $\liegf$. Let $\sigma$ denote a Coxeter element of $\liegf$. It is the element of $GL(\lieh)$ given by the product of all simple reflections (in some order). It has order $\coxn$, the Coxeter number of $\liegf$. Since $\sigma$ is a Weyl group element, we may lift it to an inner automorphism of $\liegf$, of order $\coxn$.  We also denote this automorphism as $\sigma$.
We have another finite order inner automorphism of $\liegf$, defined by:
\[\zeta = \exp \left(\ad \frac{2\pi i \rho^\vee}{\coxn}\right) = \Ad \left( \exp \frac{2\pi i \rho^\vee}{\coxn} \right), \]
where $\rho^\vee \in \liehf$ is the unique element such that $\alpha_i(\rho^\vee)=1$ for all simple roots $\alpha_i$ of $\liegf$.  It is well-known \cite[Prop. 3.4, Remark (e)]{kac-infdim-dedekind} that the automorphisms $\sigma$ and $\zeta$ are conjugate under an inner automorphism $\psi$ of $\liegf$, i.e., $\psi \sigma \psi^{-1} = \zeta$. 
Since they are all inner, $\sigma, \zeta, \psi$ lift to automorphisms $\widetilde{\sigma}, \widetilde{\zeta}, \widetilde{\psi}$ of $\lattva[Q]$ via \eqref{eq:lift_of_inner}. The automorphisms $\widetilde{\sigma}$ and $\widetilde{\zeta}$ also have order $\coxn$. Let $u = (h_1t^{j_1}) (h_2 t^{j_2}) \cdots (h_kt^{j_k}) \in \fock$ with $h_i \in \liehf$ and $j_i <0$ for $1 \leq i \leq k$ (see \S\ref{sec:heisva}).  The action of  $\widetilde{\sigma}$ and $\widetilde{\zeta}$  on $u \otimes e^\alpha \in \lattva$ is given by \cite{bakalov-milanov,dong-mason}:
\begin{align*}
  \widetilde{\sigma}(u \otimes e^\alpha) &=   (\sigma(h_1) t^{j_1}) \, (\sigma(h_2) t^{j_2}) \,\cdots \,(\sigma(h_k) t^{j_k}) \otimes e^{\sigma\alpha}, \\ \widetilde{\zeta}(u \otimes e^\alpha) &= e^{2\pi i \alpha(\rho^\vee)/\coxn} \, (u \otimes e^\alpha).
\end{align*}
Finally, we observe that since $\psi \sigma \psi^{-1} = \zeta$, the automorphisms  $\widetilde{\psi} \,\widetilde{\sigma} \,\widetilde{\psi}^{-1}$ and  $\widetilde{\zeta}$ agree on $\vqgrad[1]$. Since $\vqgrad[1]$ strongly generates $V_Q$, we obtain $\widetilde{\psi} \,\widetilde{\sigma} \,\widetilde{\psi}^{-1} = \widetilde{\zeta}$ on $V_Q$.

\subsection{The $\sigma$-twisted realization of the vacuum module}\label{sec:princ-voc}

There is a $\widetilde{\sigma}$-twisted representation $M_{\sigma}$ of the lattice vertex algebra $\lattva$ such that the associated representation of $\lieg \cong \lieg[\sigma]$ is isomorphic to the vacuum module $\vacm$. This is called the principal vertex operator realization \cite{Kac-Kazhdan-Lepowsky-Wilson}.
The linear automorphism $\sigma$ of $\liehf$ leaves the form on $\liehf$ invariant. We have
\[M_\sigma=\fock[\sigma],\]
the corresponding $\sigma$-twisted Fock space obtained by taking $\varphi = \sigma$ in \eqref{eq:sigtwistfock}. The action of $\fock \subset \lattva$ on $M_\sigma$  given by \eqref{eq:twistatefield} extends to an action of all of $\lattva$ \cite[(3.11)]{bakalov-milanov}.

\subsection{The $\zeta$-twisted realization of the vacuum module} \label{sec:zeta-voc}
Similarly, there is a $\widetilde{\zeta}$-twisted representation $M_{\zeta}$ of $\lattva[Q]$ which is isomorphic to $\vacm$ as a $\lieg \cong \lieg[\zeta]$-module. This is given by $M_\zeta = V_{Q + \rho/\coxn}$, i.e., $M_\zeta$ coincides with $V_Q = \fock \otimes_\complex \twistga$ as a vector space, with the action of $V_Q$ on $M_\zeta$ given by the $\rho/\coxn$-shifted versions of the vertex operators in equations~\eqref{eq:heis-action-lattva} and \eqref{eq:Q-action-lattva}. To define these, we let the operators $h_n, \, n \neq 0$ and $e^\alpha, \, \alpha \in Q$ act on $V_Q$ in the same manner as in \S\ref{sec:rl-va}, and redefine the actions of $h_0$ and $z^{\alpha_0}$ as follows:
\begin{align}
  h_0 (u \otimes e^\beta) &= (\beta + {\rho}/{\coxn}) (h) \, (u \otimes e^\beta) \label{eq:zetatwisted1}\\
  z^{\alpha_0} (u \otimes e^\beta) &= \form{\alpha}{\beta + \rho/\coxn} \, (u \otimes e^\beta) \label{eq:zetatwisted2}
  \end{align}
for all $h \in \liehf, \alpha, \beta \in Q, u \in \fock$.
  With these definitions, the vertex operators in \eqref{eq:heis-action-lattva} and \eqref{eq:Q-action-lattva} make $M_\zeta$ into a $\widetilde{\zeta}$-twisted $\lattva$-module \cite{lepowsky-calctwistvo,dong-mason}.\footnote{There is a  sign error in \cite[Theorem 3.4 (1)]{dong-mason}  (it should be $V_{L-\beta}$) which is corrected in part (2) of that theorem. Their convention on twisted modules and our (now standard) convention differ by a sign.} Viewed as a $\lieg \cong \lieg[\zeta]$-module, it is isomorphic to a level 1 irreducible highest weight representation \cite{lepowsky-calctwistvo}. That the highest weight is  $\Lambda_0$ can be readily seen from the definition of the shifted vertex operators and the isomorphism $\lieg \stackrel{\sim}{\to} \lieg[\zeta]$.

\subsection{Isomorphism of the two realizations of the vacuum module}
Now $M_\sigma \cong \vacm$ as a $\lieg[\sigma]$-module and $M_\zeta \cong \vacm$ as a $\lieg[\zeta]$-module. In other words $M_\sigma$ and $M_\zeta$  are isomorphic as $\lieg \cong \lieg[\sigma]  \cong \lieg[\zeta]$ modules. From the discussion of \S\ref{sec:twist1}-\S\ref{sec:twist3}, this implies that there is a map (cf., \cite[Theorem 4.5.2]{FLMbook} for the $\afsl[2]$-case):
$  {\Psi}: M_{{\sigma}} \to M_{{\zeta}} $
such that given $k \in \rationals$ and $v \in M_{{\sigma}}$
\begin{equation} \label{eq:equivariance}
  {\Psi}(X_{(k)} \bullet_{\sigma} v) = \left(\widetilde{\psi}(X)\right)_{(k)} \bullet_{\zeta} {\Psi}(v) \text{ for all } X \in \vqgrad[1].
\end{equation}
Here $\bullet_{\sigma}$ and $\bullet_{\zeta}$ indicate the $\lattva$-actions on $M_\sigma$ and $M_\zeta$ respectively.
Since $\vqgrad[1]$ strongly generates $\lattva[Q]$, equation~\eqref{eq:equivariance} holds for all $X \in \lattva[Q]$.  
\begin{equation} \label{eq:equivariance-all}
 {\Psi}(X_{(k)} \bullet_{\sigma} v) = \left(\widetilde{\psi}(X)\right)_{(k)} \bullet_{\zeta} {\Psi}(v) \text{ for all } X \in \lattva[Q].
\end{equation}
Equivalently:
\begin{equation} \label{eq:equivariance-all-fields}
  {\Psi} \; Y_{M_\sigma}(X; z) \; {\Psi}^{-1} = Y_{M_\zeta}(\widetilde{\psi}(X); z)   \text{ for all } X \in \lattva[Q]. 
\end{equation}

\section{The Brylinski filtration on $\vacm$}\label{sec:bryl-fil}
\subsection{Definition}\label{sec:bryl-def}
Let $\lieg$ denote a simply-laced affine Lie algebra, with standard Chevalley generators $e_i, f_i, \alphac_i \; (0 \leq i \leq \ell)$, canonical central element $K$ and derivation $d$. The {\em principal gradation} of $\lieg$ is defined by setting $\deg e_i = 1 = - \deg f_i$ and $\deg \alphac_i = \deg K = \deg d =0 $. The principal Heisenberg subalgebra $\prinheis$ of $\lieg$ is defined as:
\[
\prinheis = \{x \in \lieg: [x,e] \in \complex K\},
\]
where $e=\sum_{i=0}^\ell e_i$. It is isomorphic to an infinite-dimensional Heisenberg Lie algebra \cite[Lemma 14.4]{kac}. Since $e$ is a homogeneous element, it follows that $\prinheis$ is graded with respect to the principal gradation of $\lieg$:
$ \prinheis = \bigoplus_{j \in \integers} \prinheis_j $
with $\prinheis_0 = \complex K$. We set $\prinheis^+$ (respectively $\prinheis^-$) to be $\bigoplus_{j>0} \prinheis_j$ (respectively $\bigoplus_{j<0} \prinheis_j$). The multiset comprising the elements $j \neq 0$ each repeated $\dim \prinheis_j$ times is called the multiset of exponents of the affine Lie algebra $\lieg$ \cite[Chapter 14, Table E]{kac}.

The subspaces $F^i \vacm$ of the Brylinski filtration  \cite{slofstra} are defined (for $i \geq -1$) by:
\begin{equation*}
F^i \vacm = \{ v \in \vacm: x^{i+1} v=0 \text{ for all } x \in \prinheis^+\}.
\end{equation*}
Since $\prinheis^+$ is abelian, we have $v \in F^i \vacm$ iff 
\begin{equation} \label{eq:bryl-def}
   x_1x_2\cdots x_{i+1}v=0 \text{ for all tuples } (x_1, \cdots, x_{i+1}) \text{ of elements from }\prinheis^+.
\end{equation}

\subsection{In the principal vertex operator realization}\label{sec:prinvo-bryl}
It is easiest to describe the action of $\prinheis$ on $\vacm$ in the principal vertex operator realization $M_\sigma$ of \S~\ref{sec:princ-voc}. We  identify $\vacm$ with the space $M_\sigma$ via the isomorphism of $\S~\ref{sec:princ-voc}$. Then, the action of $\prinheis$ on $M_\sigma$ is given by the modes of the fields:
\[ Y_{M_\sigma}(ht^{-1}, z) = \sum_{j \in \frac{1}{\coxn} \integers} h_{(j)} (\sigma) \, z^{-j-1},\]
where $h \in \liehf$. 
In particular, the image of $\prinheis^+$ (respectively $\prinheis^-$) under the natural map $\lieg \to \End M_\sigma$ is the span of
$\{ h_{(j)}(\sigma): h \in \liehf, j >0 \}$ (respectively $\{ h_{(j)}(\sigma): h \in \liehf, j <0 \}$). We recall that the Coxeter element $\sigma$ acts on $\liehf\backslash\{0\}$ without fixed points, so $h_{(j)}(\sigma)=0$ for all $j \in \integers, h \in \liehf$.

We regard the Brylinski filtration as being defined on $M_\sigma$. Now $M_\sigma$ is a $\sigma$-twisted $V_Q$ module, and since $\fock$ is a $\sigma$-invariant vertex subalgebra of $V_Q$, we may view $M_\sigma$ as a $\sigma$-twisted $\fock$-module.
We recall from \S\ref{sec:heisva} that $\fock =\bigoplus_{d \geq 0} \fock^{[d]}$, with $\fock^{[1]} = \liehf \otimes t^{-1}$ identified with $\liehf$ . The following lemma concerns the action of $\fock$ on the subspaces $F^i M_\sigma$ of the Brylinski filtration on $M_\sigma$.

\begin{lemma}\label{lem:fockmodes-bryl}
With notation as above:
\be
\item Let $h \in \liehf, j \in \frac{1}{\coxn} \integers$ and $i \geq 0$.
  \be
\smallskip\item If $j>0$, then $h_{(j)}(\sigma)$ maps $F^i M_\sigma$ to $F^{i-1} M_\sigma$.
\smallskip  \item If $j<0$, then $h_{(j)}(\sigma)$ maps $F^i M_\sigma$ to  $F^{i+1} M_\sigma$.
\ee
\medskip\item Let $d \geq 0, j \in \frac{1}{\coxn} \integers$ and $i \geq 0$. If $X \in \fock^{[d]}$, then $X_{(j)}(\sigma)$ maps $F^i M_\sigma$ to $F^{i+d} M_\sigma$, where $X_{(j)}(\sigma)$ are the modes of the field $Y_{M_\sigma}(X,z)$.
\ee
\end{lemma}
\begin{proof}
It is clear that 1(a) follows directly from the definition \eqref{eq:bryl-def}. For 1(b), we proceed by induction on $i$, starting with $i=-1$ (where it holds trivially). Let $h' \in \liehf, p \in \frac{1}{\coxn} \integers$ with $p >0$ and $v \in F^i M_\sigma$, $i \geq 0$. We have \cite[\S 3.3]{bakalov-milanov}:
\begin{equation}\label{eq:hhprime}
h^{\prime}_{(p)}(\sigma) \,h_{(j)}(\sigma) v =  h_{(j)}(\sigma) h'_{(p)}(\sigma) v + p\,\delta_{p,-j}\form{h'}{h} v,
\end{equation}
where $\form{\cdot}{\cdot}$ is the standard invariant form on $\liehf$. By 1(a), $h'_{(p)}(\sigma)v \in F^{i-1}M_\sigma$. The induction hypothesis implies that
$h_{(j)} (\sigma) h'_{(p)}(\sigma) v \in F^i M_\sigma$. Thus the right hand side of \eqref{eq:hhprime} is in $F^i M_\sigma$. Equation~\eqref{eq:bryl-def} completes the proof.

For (2), we recall from \S\ref{sec:heisva}  that in the vertex algebra $\fock$, each $X \in \fock^{[d]}$ is a linear combination of  terms of the form
$\left(h^1 \otimes t^{j_1}\right) \cdots \left(h^k \otimes t^{j_k}\right) = h^1_{(j_1)} \cdots h^k_{(j_k)} \vac$ with $h^i \in \liehf$, $j_i <0, \, \sum j_i = -d$ and where $\vac$ denotes the vacuum vector of $\fock$. In particular, this implies $k \leq d$. An iterated application of the product identity \eqref{eq:prodidentity} implies that each mode $X_{(j)}(\sigma)$ is a linear combination of terms of the form:
\[
h^{1}_{(p_1)}(\sigma)\, h^{2}_{(p_2)}(\sigma)\, \cdots h^{k}_{(p_k)}(\sigma) \in \End M_\sigma
\]
with $p_i \in \frac{1}{\coxn} \integers$. The assertion now follows from part (1) and the observation that $k \leq d$.
\end{proof}

\section{The $\walg$-algebra}\label{sec:walg}

We freely use the notations of the preceding sections.
\subsection{Definition}
Let $\lieg$ be a simply-laced affine Lie algebra. The  $\walg$-algebra associated to $\lieg$ is the vertex subalgebra of the lattice vertex algebra $\lattva[Q]$ defined by:
\begin{equation}\label{eq:walgdef}
  \walg := \bigcap_{X \in \vqgrad[1]} \ker_{{}_{\lattva[Q]}} \, X_{(0)}.
\end{equation}
The conformal vector $\confvec$ of $\walg$ coincides with that of $\lattva$ and is given by \eqref{eq:conflva} \cite[\S 2.4]{bakalov-milanov}. It induces a $\integers_+$-grading on $\walg$. 

\medskip
Let $d_1 \leqslant d_2 \leqslant \cdots \leqslant d_\ell$ be the list of degrees (exponents plus one) of $\liegf$.
We recall from the introduction (Theorem~\ref{thm:ff}) that there exist elements $\omd[i] \in \walg$ of degree $d_i$ such that $\walg$ is freely generated by  $\omd[1],\, \omd[2],\, \cdots,\, \omd[\ell]$. Further, $d_1=2$ and the conformal vector $\confvec$ is precisely $\omd[1]$.

\begin{remark}\label{rem:walg-ginv}
We recall from \S\ref{sec:innauts} that  $\vqgrad[1] \cong \liegf$ as Lie algebras, with $X_{(0)}$ corresponding to $\ad X$. We can identify $\lattva$ with the homogeneous vertex operator realization of $\vacm$ \cite{frenkel-kac}. It follows from \eqref{eq:walgdef} that  $\walg$ maps to the subspace $\vacm^{\liegf}$ of $\liegf$-invariants under this identification (cf. \cite[Remark 2.1]{bakalov-milanov}).
\end{remark}

\subsection{Twisted $\lattva$-modules restrict to ordinary $\walg$-modules}

Since $\bigcap_{\,h \in \liehf} \,\ker_{{}_{\lattva[Q]}}  h_{(0)}$ is the Heisenberg vertex algebra $\fock \subset \lattva$, it follows that $\walg$ is a vertex subalgebra of $\fock$.
In addition, the $\walg$-algebra is pointwise fixed by any inner automorphism of $\lattva[Q]$. This is because an inner automorphism (\S \ref{sec:innauts}) is a product of automorphisms of the form $\exp X_{(0)}$ for $X \in \liegf = \vqgrad[1]$. Since each $X_{(0)}$ annihilates $\walg$, we have $\exp X_{(0)}$ acts as the identity operator on $\walg$. 

 In particular, given a $\varphi$-twisted representation $M$ of $\lattva[Q]$, where $\varphi$ is a finite order inner automorphism of $\lattva[Q]$, its restriction to $\walg$ defines an ordinary (untwisted) representation of $\walg$ on $M$.
 Now applying this to equation \eqref{eq:equivariance-all}, we conclude $\widetilde{\psi}(X) = X$ for $X \in \walg$, and hence that:
  \[ {\Psi}(X_{(k)} \bullet_{\sigma} v) = X_{(k)} \bullet_{\zeta} {\Psi}(v) \]
  for $X \in \walg,\, k \in \integers$ and $v \in M_{{\sigma}}$. We have thus proved:

\begin{proposition}
  The map $\Psi: M_{{\sigma}} \to M_{{\zeta}}$ is an isomorphism of  $\walg$-modules.
\end{proposition}  

\subsection{The Zhu algebra of $\walg$}
We recall the definition of the Zhu algebra $\zhu$ \cite{zhu}. Given $a \in \wgrad[d]$, we let $\deg a:=d$ and write $Y(a,z) = \sum_n a_n\, z^{-n-d}$. Consider the subspace $O(\walg)$ spanned by the elements of the form $\sum_{p=0}^{\deg a} \binom{\deg a}{p} a_{-1-p}\,b$ for  homogeneous elements $a, b$ of $\walg$. The Zhu algebra is the quotient $\zhu = \walg/O(\walg)$ equipped with the associative multiplication:
\[
a * b = \sum_{p=0}^{\deg a} \binom{\deg a}{p} a_{-p}\,b \; \pmod{O(\walg)}.
\]
Let $M = \bigoplus_{c \in \complex} M_c$ be a graded $\walg$-module and let $\mtop$ denote the sum of the nonzero homogeneous subspaces $M_c$ for which $M_{c+n} =0$ for all $n >0$ \cite[\S 3.12]{arakawa}. Then $\mtop$ is a $\zhu$-module with the action:
\[ a \cdot m = a_0 m, \]
where $m \in \mtop$ and $a$ is the image in $\zhu$ of a homogeneous element of $\walg$.

It is well-known that $\zhu$ is isomorphic to $\centug$, the center of the universal enveloping algebra of $\liegf$ \cite{arakawa, frenkelkacwaki}. We identify $\zhu$ with $\centug$ using the isomorphism of Arakawa \cite[Theorem 4.16.3 (ii)]{arakawa}. With this identification, the one-dimensional representations of $\zhu$ are given by the central characters:
\[ \cch : \centug \to \complex \]
for $\lambda \in \liehf^*$; each $z \in \centug$ acts as the scalar $\cch(z)$ on the Verma module $M(\lambda)$ of $\liegf$ with highest weight $\lambda$.

\subsection{Verma modules of $\walg$, character, PBW basis}\label{sec:wvermazhu}

We point the reader to Arakawa \cite[\S 5.1]{arakawa} for the full definition of Verma modules of $\walg$-algebras. Here, we will content ourselves with recalling their essential properties.

Given a central character $\cch$ of $U\liegf$, the Verma module $\vermaw(\cch)$ is a graded $\walg$-module, with $\mtop[\vermaw(\cch)] = \complex \vac[\cch]$, where $\vac[\cch]$ is a cyclic vector of $\vermaw(\cch)$. The Zhu algebra $\zhu$ acts on $\vac[\cch]$ by
\[ z \vac[\cch] = \cch(z) \, \vac[\cch] \; \text{ for all } z \in \zhu.\]

Further, given a graded $\walg$-module $M$ and a nonzero vector $m \in \mtop$ such that $z m = \cch(z) m$ for all $z \in \zhu$, there exists a unique $\walg$-homomorphism $\psi: \vermaw(\cch) \to M$ sending $\vac[\cch] \mapsto m$.\footnote{Our $\walg$ corresponds to $\walg_1$ in the notation of \cite{fkrw}. In Arakawa's notation \cite{arakawa}, this corresponds to $k+\coxdual =1$, where $\coxdual$ is the dual coxeter number. In this case, the eigenvalue of $\omd[]_0$ (where $\omd[] = \omd[1]$ is the conformal vector of $\walg$) on $\vac[\gamma_\mu]$ is  $\Delta_{\mu} := |\mu+\overline{\rho}|^2/2$ as per \cite[(286)]{arakawa}. This is the top degree of $\vermaw(\cch)$.}
If $M$ is a $\walg$-module on which $\omd[]_0$ acts semisimply with finite-dimensional eigenspaces, we define its character by $\ch M = \tr_M(q^{\omd[]_0})$.
Our definition  differs from that of \cite[\S 5.6]{arakawa} by a normalization factor.\footnote{
Arakawa \cite{arakawa} defines it as :  $q^{-c(k)/24} \tr_M(q^{\omd[]_0})$. For $k+\coxdual =1$, $c(k) = \ell$.}
Correcting for this, the character of a Verma module is given by \cite[Proposition 5.6.6]{arakawa}:
\begin{equation}\label{eq:vermachar}
  \ch \vermaw(\cch) = \frac{q^{|\lambda+\rho|^2/2}}{\varphi(q)^\ell},
\end{equation}
where $\rho$ is the Weyl vector of $\liegf$ and $\varphi = \prod_{n \geq 1} (1-q^n)$ is Euler's function. 

  Further, the Verma module $\vermaw(\cch)$ has a PBW type basis. More precisely,
  as established by Arakawa \cite[Prop 5.1.1]{arakawa}, there exist filtrations $F^\bullet \walg$ on $\walg$ and $F^\bullet \vermaw(\cch)$  on the Verma module such that $\gr^F \vermaw(\cch)$ is isomorphic to the polynomial algebra $\complex[\omd[p]_{k}: \, 1 \leq p \leq \ell, \, k \leq -1]$ as a $\gr^F \walg$-module.
In particular, consider the following elements of $\vermaw(\cch)$:
\beq\label{eq:wpbw}
\omd[p_1]_{k_1}\; \omd[p_2]_{k_2}\; \cdots\; \omd[p_r]_{k_r} \,\vac[\cch],
\eeq
\smallskip
where (i) $r \geqslant 0$\quad(ii) $\ell \geqslant p_1 \geqslant p_2 \geqslant \cdots \geqslant p_r \geqslant 1$ \quad(iii) $k_j \leqslant -1$ for all $j$\quad   and (iv) if $p_i = p_{i+1}\,$, then $k_i \leqslant k_{i+1}$.
Proposition 5.1.1 of \cite{arakawa} implies that
\[ \{\gr^F v: v \text{ is of the form } \eqref{eq:wpbw} \} \]
is a basis of $\gr^F \vermaw(\cch)$. This in turn implies that the vectors of the form \eqref{eq:wpbw} form a basis of $\vermaw(\cch)$.

\section{The space $Z$ and $M_\zeta$}\label{sec:zmzeta}

Consider the subspace $Z:=\vacm^{\liehf}$ of $\liehf$-invariants of $\vacm$:
\[ Z= \bigoplus_{n \geqslant 0} Z_n = \bigoplus_{n \geqslant 0} L(\Lambda_0)_{\Lambda_0 - n \delta}. \]
Now consider the $\zeta$-twisted module $M_\zeta$ of $\lattva[Q]$. Under the $\lieg$-module isomorphism $\vacm \to M_{\zeta}$ of \S\ref{sec:zeta-voc}, the space $Z$ maps to $\fock \otimes e(\rho/\coxn)$.
 Since $\widetilde{\zeta}$ clearly fixes every element of the Heisenberg vertex subalgebra $\fock$ of $\lattva[Q]$, the restriction of $M_\zeta$ to $\fock$ is an ordinary (untwisted) representation of $\fock$. It is clear from \eqref{eq:zetatwisted1} that the subspace $Z = \fock \otimes e(\rho/\coxn) \subset M_\zeta$ is $\fock$-invariant.
For $u \in \fock^{[p]}$, let the corresponding field be denoted (renumbered with conformal weight)
$Y_{M_\zeta}(u,z) = \sum_{n \in \integers} u_n \,z^{-n-p}$. We then have
\begin{equation} \label{eq:modejumps}
u_{-k} (Z_m) \subset Z_{m+k} \text{ for all } k, m \in \integers.
\end{equation}
Viewed as a representation of the Heisenberg Lie algebra
$ \liehf \otimes \complex[z,z^{-1}] \oplus \complex c,$
the space $Z$ is isomorphic to the irreducible highest weight representation with highest weight $\rho/\coxn$.

\section{$Z$ is an irreducible $\walg$-module}\label{sec:fkrw}

\subsection{}

In the remainder of the paper, we consider $\lieg = A_\ell^{(1)}$. Let $\walg$ be the $\walg$-algebra of $\lieg$; it is a vertex subalgebra of $\fock$. Let $\lambda \in \liehf^*$ and let $\pi_\lambda$ denote the corresponding irreducible representation of $\fock$, with a highest weight vector $\vac[\lambda]$ (\S\ref{sec:oscrep}). We restrict $\pi_\lambda$ to a representation of $\walg$.

\begin{proposition}\label{prop:propa}
Let $\lambda \in \liehf^*$ and consider $\pi_\lambda$ as a $\walg$-module as above. If $\pi_\lambda$ is irreducible (as a $\walg$-module), then it is isomorphic to a Verma module of $\walg$.
\end{proposition}

\begin{proof}
  Firstly, $\pi_\lambda$ is a graded $\walg$-module with $\mtop[(\pi_\lambda)] = \complex \vac[\lambda]$. By \S\ref{sec:wvermazhu}, this space is a one-dimensional $\zhu$-module, and is thus given by a central character of $U\liegf$. In other words, there exists $\mu \in \liehf^*$ such that $z \vac[\lambda] = \gamma_\mu(z) \vac[\lambda]$ for all $z \in \zhu \cong \centug$. Since $\vac[\lambda]$ is of top degree, we also have $\omd[p]_n \vac[\lambda] =0$ for all $1 \leq p \leq \ell$ and $n >0$.

Again by \S\ref{sec:wvermazhu}, there is a homomorphism of $\walg$-modules:
  \[ \phi: \vermaw(\gamma_\mu) \to \pi_\lambda \text{ sending } \vermahigh \to \vac[\lambda]. \]
  The hypothesis that $\pi_\lambda$ is irreducible implies that $\phi$ is surjective. We now compare the characters of these two modules. By \eqref{eq:vermachar}, we have
  $\ch \vermaw(\gamma_\mu)  = \frac{q^{|\mu+\rho|^2/2}}{\varphi(q)^\ell}$. Let $\ch \pi_\lambda =  \tr_{\pi_\lambda}(q^{\omd[]_0})$. Now $\walg$ is a vertex subalgebra of $\fock$ and their conformal vectors coincide (and are given by \eqref{eq:conflva}).  We recall from \S\ref{sec:oscrep} that the character 
\[ \tr_{\pi_\lambda}(q^{\omd[]_0}) = \frac{q^{|\lambda|^2/2}}{\varphi(q)^\ell}. \]
Since $\omd[]_0 \vermahigh = \frac{|\mu+\rho|^2}{2} \vermahigh$ and $\omd[]_0 \vac[\lambda] = \frac{|\lambda|^2}{2} \vac[\lambda]$, we conclude
$|\mu+\rho|^2 = |\lambda|^2$ and hence that $\ch \vermaw(\gamma_\mu) = \ch \pi_\lambda$. This proves that $\phi$ is an isomorphism.
\end{proof}

We have the following key theorem.

\begin{theorem}\label{thm:fkrw-thm}
  Let $\lambda \in \liehf^*$ be such that ${\lambda}({\alpha^\vee}) \not\in \integers$ for all roots $\alpha$ of $\liegf$. Then $\pi_\lambda$ is an irreducible $\walg$-module.
\end{theorem}

For the proof, we shall use the main theorem of \cite{fkrw} which relates representations of the Lie algebra $\woneinf$ with those of the $\walg$-algebra of $\gln \complex$. We recall the relevant notations and results in the next two subsections.

\subsection{The $\walg$-algebra of $\gln$}\label{sec:wgl}
The vertex algebra $\walg(\gln)$ is defined analogously to $\walg(\sln)$. Let $\liehf(\gln)$ denote the set of diagonal matrices in $\gln\complex$ and let $\varepsilon_i \in \liehf(\gln)^*$ be defined by $\varepsilon_i(H) = H_{ii}$ for each diagonal matrix $H$. Consider the integral lattice $\Qtil = \sum_i \integers \varepsilon_i$ with bilinear form $\form{\varepsilon_i}{\varepsilon_j} = \delta_{ij}$. The lattice vertex algebra $\lattva[\Qtil]$ is defined as in \S\ref{sec:rl-va}:
\[ \lattva[\Qtil] = \focktil \otimes_\complex \twistga[\Qtil], \]
where the Fock space $\focktil$ is the symmetric algebra on the space $\sum_{j < 0} \; \liehf(\gln) \otimes t^j$ and $\varepsilon$ is a certain bimultiplicative cocyle on $\Qtil$. The state-field correspondence is also analogous to that in \S\ref{sec:rl-va} (see \cite[Theorem 5.5]{kac-vafb} for details). The lattice vertex algebra is now $\frac{1}{2} \integers$-graded, with the grade 1 piece $\vqtilgrad[1]$ being spanned by   $ht^{-1} \otimes 1$ and $1 \otimes e^\alpha$ for $h \in \liehf(\gln), \alpha  \in \roots$, where $\roots = \{\varepsilon_i - \varepsilon_j: i \neq j\}$. We identify $\vqtilgrad[1]$ with $\gln$.

Define $\walg(\gln)$  to be the vertex subalgebra of $\lattva[\Qtil]$ given by:
\[ \walg(\gln) := \bigcap_{X \in \vqtilgrad[1]} \ker_{{}_{\lattva[\Qtil]}} \, X_{(0)}. \]
There is a one-parameter family of conformal vectors of $\walg(\gln)$ defined by (\cite[\S 4]{fkrw}):
\begin{equation}\label{eq:wgl-confvec}
\omega_a =  a I(-2) \vac + \frac{1}{2} \sum_{i=1}^{\ell+1} u_i(-1) u^i(-1) \vac, 
\end{equation}
where $a \in \complex$, $\{u_i\}, \{u^i\}$ are dual bases of $\liehf(\gln)$ with respect to the form $\form{\cdot}{\cdot}$, $I$ is the identity matrix, $\vac = 1\otimes 1$ is the vacuum vector of $\lattva[\Qtil]$ and  $h(k)$ is the operator $ht^k \otimes \id$ on $\lattva[\Qtil]$ for  $h \in \liehf(\gln), \, k \in \integers$.

\subsection{$\gln$ vs $\sln$}\label{sec:glvssl}
Since $\liehf(\gln) = \liehf(\sln) \oplus \complex I$, where $I$ is the identity matrix, the corresponding Fock spaces are related by:
$ \focktil = \fock \otimes \fock[1]$ (tensor product of vertex algebras), 
where $\fock[1]$ is the symmetric algebra on  $\sum_{j < 0} \; I \otimes t^j$. 
It is easily observed \cite[Remark 4.2]{fkrw} that
\[\walg(\gln) \cong \walg(\sln) \otimes \fock[1] \;\;\subset \fock \otimes \fock[1].\]

Given $\lambda \in \liehf(\sln)^*$, we extend it to $\lamtil \in \liehf(\gln)^*$ by defining $\lamtil(I) =0$. 
Let $\pi_\lambda, \, \pi_{\lamtil}$ denote the corresponding irreducible  representations of $\fock$ and $\focktil$ as in \S\ref{sec:oscrep}. Viewing $\fock[1]$ as a module over itself, we have:
\[ \pi_{\lamtil} \cong \pi_\lambda \otimes \fock[1] \]
as $\fock \otimes \fock[1]$-modules, and hence by restriction as $\walg(\sln) \otimes \fock[1]  \cong \walg(\gln)$-modules.

Now $\fock[1]$ is an irreducible module over itself. Hence, by \cite[Proposition 4.7.2]{frenkel-huang-lepowsky} (whose mild technical conditions hold in this case), we conclude that $\pi_{\lamtil}$ is an irreducible $\walg(\gln)$-module if and only if $\pi_\lambda$ is an irreducible $\walg(\sln)$-module.

\subsection{The Lie algebra $\woi$}\label{sec:woneinfgl}
In this subsection, we recall the definition and key properties of the Lie algebra $\woi$, following \cite{fkrw}. Let $\scrD \subset \End_\complex \,\complex[t,t^{-1}]$ denote the Lie algebra of regular differential operators on $\complex^\times$, with the usual bracket. Each of the following collections forms a basis of $\scrD$:
\begin{enumerate}
\item $J^{\ell}_k = - t^{\ell+ k} (\partial_t)^k$ with $k, \ell \in \integers$, $k \geq 0$,
\smallskip\item $L^{\ell}_k = -t^\ell (t \partial_t)^k$ with $k, \ell \in \integers$, $k \geq 0$,
\end{enumerate}
where $\partial_t = \frac{\partial}{\partial t}$. This Lie algebra has a $\complex$-valued 2-cocyle $\psi$ given by:
\[
\psi(f(t) \partial_t^m, g(t) \partial_t^n) = \frac{m! \,n!}{(m+n+1)!} \Res\, (\partial_t^{n+1}f(t)) \,(\partial_t^m g(t)),
\]
where as usual, for a Laurent polynomial $f \in \complex[t,t^{-1}]$, $\Res f(t)$ denotes the coefficient of $t^{-1}$ in $f(t)$. We let $\woi = \scrD \oplus \complex C$ be the one-dimensional central extension of $\scrD$ defined by the cocyle $\psi$, i.e., with Lie bracket defined by
$[X,Y]:=[X,Y]_{\scrD} +\psi(X,Y) \,C$ for all $X, Y \in \scrD$.

Consider the Lie subalgebra $\scrD[P] = \mathrm{span} \{J^{\ell}_k: \ell + k \geq 0\}$ of $\scrD$ and let $\widehat{\scrD[P]} = \scrD[P] \oplus \complex C \subset \woi$. Given $c \in \complex$, we form the induced $\woi$-module
\[ M_c = U\woi \otimes_{U\widehat{\scrD[P]}} \,\complex, \]
where $C$ acts as $c$ on the one-dimensional space $\complex$ and $\scrD[P]$ acts as zero. The module $M_c$ admits a unique irreducible quotient $V_c$. We recall that $V_c$ is a vertex algebra and that $V_c$-modules may be canonically viewed as $\woi$-modules on which $C$ acts by the scalar $c$. Further, $V_c$ has a one-parameter family of conformal vectors \cite[Theorem 3.1]{fkrw}:
\begin{equation}\label{eq:vacmod-confvec}
\omega(\beta) = (J^1_{-2} - \beta J^0_{-2}) \vac,
\end{equation}
where $\vac$ is the image of $1 \otimes 1$ in $V_c$. The central charge of the corresponding Virasoro field is $-(12\beta^2 - 12\beta +2)c$.
Now, the key fact of relevance to us is the following \cite[Theorem 5.1]{fkrw}:

\begin{proposition} (Frenkel-Kac-Radul-Wang)\label{prop:fkrw-mainprop}
There is an isomorphism of vertex algebras $V_{\ell+1} \to \walg(\gln)$ which maps the conformal vectors $\omega(\beta) \mapsto \omega_{1/2 - \beta}$ (equations~\eqref{eq:wgl-confvec} and \eqref{eq:vacmod-confvec}).\footnote{There seems to be a sign error in \cite{fkrw}, where this appears as $\omega(\beta) \mapsto \omega_{\beta-1/2}$. This does not affect our computation either way.} Hence, any representation of the vertex algebra $\walg(\gln)$ can be canonically lifted to a representation of $V_{\ell+1},$ or equivalently, to a representation of the Lie algebra $\woi$ with central charge $\ell+1$.
\end{proposition}

Given $\gamma \in \liehf(\gln)^*$, consider the representation $\pi_{\gamma}$ of $\walg(\gln)$ as in \S\ref{sec:glvssl}. Let $\vac[\gamma]$ denote its highest weight vector. Let $U(\gamma)$ denote the $\walg(\gln)$-submodule of $\pi_{\gamma}$ generated by $\vac[\gamma]$. This is a cyclic, graded module and therefore has a unique irreducible quotient $V(\gamma)$. The following is one of the main results of \cite{fkrw}:

\begin{proposition}\label{prop:lift} \em{(\!\!\cite[Prop 5.1]{fkrw})}
The lifting of $V(\gamma)$ to a $\woi$-module is isomorphic to the primitive $\woi$-module with exponents $\gamma(E_{ii}): 1 \leq i \leq \ell+1$.
\end{proposition}

We refer the reader to \cite{fkrw} for the definition of {\em primitive $\woi$-modules} and {\em exponents}. The character of such modules was also determined by Frenkel-Kac-Radul-Wang. We now state this result in the special case of interest to us.

\begin{proposition}\label{prop:primchar}
  Let $s_i \; (1 \leq i \leq \ell+1)$ be complex numbers such that $s_i \not\equiv s_j \pmod{\integers}$ for all $i \neq j$. Then the character of the primitive $\woi$-module $M$ with exponents $\{s_i\}$ is given by:
  \begin{equation}\label{eq:trprim}
    \tr_M(q^{\lzo}) = \frac{q^{\sum_i \, s_i(s_i-1)/2}}{\varphi(q)^{\ell+1}}.
  \end{equation}
 
\end{proposition}
This follows from Proposition 2.1 and equation (2.5) of \cite{fkrw}. \qed

\subsection{Proof of Theorem~\ref{thm:fkrw-thm}} We use the notations introduced in the above subsections. We have by \S\ref{sec:glvssl} that $\pi_\lambda$ is an irreducible $\walg(\sln)$-module iff  $\pi_{\lamtil}$ is an irreducible $\walg(\gln)$-module. Since $V(\lamtil)$ is a subquotient of $\pi_{\lamtil}$, the latter assertion is equivalent to the assertion that $\pi_{\lamtil}$ and  $V(\lamtil)$ have the same character, i.e.,
\[ \tr_{\pi_{\lamtil}}(q^{\lzo}) = \tr_{V(\lamtil)}(q^{\lzo}). \]
Define $s_i = \lamtil(E_{ii})$. The hypothesis that ${\lambda}({\alpha^\vee}) \not\in \integers$ for all positive roots $\alpha$ of $\sln\complex$ implies that $s_i - s_j \not\in \integers$ for all $i \neq j$. Propositions~\ref{prop:lift} and \ref{prop:primchar} then imply that the character of $V(\lamtil)$ is given by the right hand side of equation~\eqref{eq:trprim}.

To compute the character of $\pi_{\lamtil}$, we need to identify $\lzo \in \woi$ with the appropriate conformal vector $\confvec_a$ of $\walg(\gln)$. We recall from Proposition~\ref{prop:fkrw-mainprop}  that under the isomorphism of vertex algebras $V_{\ell+1} \stackrel{\sim}{\to} \walg(\gln)$, their conformal vectors correspond thus:
\[ \omega(\beta) \mapsto \omega_{1/2 - \beta}.\]
Taking $\beta=0$, we obtain $\omega(0) \mapsto \omega_{1/2}$, and in particular, their zeroth modes correspond to each other. We have $\omega(0)_0 = \lzo$ and
\[(\omega_{1/2})_0 \vac[\lamtil] = \frac{\form{\lamtil}{\lamtil}}{2} - \frac{\lamtil(I)}{2} = {\sum_i \, s_i(s_i - 1)/2}. \]
So, by \S\ref{sec:oscrep}, the character of $\pi_{\lamtil}$ becomes:
\[ \tr_{\pi_{\lamtil}}(q^{\lzo}) = \tr_{\pi_{\lamtil}}(q^{(\omega_{1/2})_0}) = \frac{q^{\sum_i \, s_i(s_i -1)/2}}{\varphi(q)^{\ell+1}}. \]
This matches up correctly with the expression in \eqref{eq:trprim}, thereby completing the proof. \qed

\subsection{A basis of $Z$}
Let $\walg:=\walg(\sln)$ and regard the space $Z$ as a $\walg$-module as before. 
\begin{proposition}\label{prop:zverma}
$Z$ is isomorphic to a Verma module of $\walg$. Further, the set of vectors  in \eqref{eq:ffbasis-twisted} forms a basis of $Z$.
\end{proposition}
\begin{proof}
  The first part follows by applying Proposition~\ref{prop:propa} and Theorem~\ref{thm:fkrw-thm} to the subspace $Z$ identified with $\fock \otimes e(\rho/\coxn) \subset M_\zeta$. The element $\lambda$ in this case is just $\rho/\coxn$ which clearly satisfies the hypothesis of Theorem~\ref{thm:fkrw-thm}.
  For the second part, observe that the set of vectors in  \eqref{eq:ffbasis-twisted} is now exactly the PBW basis of the Verma module $Z$  (equation~\ref{eq:wpbw}). 

\end{proof}

\section{Brylinski-compatibility of the basis of $Z$}\label{sec:bryl-comp}
We now complete the proof of Theorem~\ref{thm:main-thm-intro}. We begin with the following simple lemma.
\begin{lemma}\label{lem:omega-modes-property}
For $1 \leq p \leq \ell$ and all $k \in \integers$, 
$\omd[p]_k(\sigma)$ maps $F^i \vacm$ to $F^{i+d_p} \vacm$.
\end{lemma}
\begin{proof}
  Since $\omd[p]$ lies in $\fock^{[d_p]}$, this is a direct consequence of Lemma~\ref{lem:fockmodes-bryl}.
\end{proof}

\subsection{Generalities on filtrations and bases}
Let $M$ be a finite-dimensional vector space with a filtration $M_0 \subset M_1 \subset M_2 \cdots\;$ such that $\bigcup M_i = M$. We consider $\gr M = \displaystyle\bigoplus_{k \geq 0} M_k/  M_{k-1}$, where $M_{-1}:=0$. For $v \in M$, we define $\gr v \in \gr M$ to be the image of $v$ under the projection $M_k \twoheadrightarrow M_k/ M_{k-1} $, where $k \geq 0$ is minimal such that $v \in M_k$. 

\begin{definition}\label{def:filt-compat-basis}
  A basis $B$ of $M$ is said to be compatible with the filtration $\{M_i\}$ if $B \cap M_i$ is a basis of $M_i$ for all $i$, or equivalently, if
$\{\gr v: v \in B\}$ is a basis of $\gr M$.
\end{definition}

It is elementary to check the equivalence of the two descriptions in the definition above.

\begin{lemma}\label{lem:filtration-grbasis}
Let $\{M_i\}_{i=0}^m$ be a filtration of $M$ as above. Suppose $B$ is a basis of $M$ and $\{B^i\}_{i=0}^m$ is a collection of pairwise disjoint subsets of $B$ such that $\bigsqcup B^i = B$. Suppose further that for all $i \geq 0$:
\be
\item $B^i \subset M_i\; ,$
\item\label{item:second} $| B^i| = \dim \left(M_i/M_{i-1}\right)$.
\ee
Then $B$ is compatible with the filtration $\{M_i\}$.
\end{lemma}

\begin{proof}
A repeated application of (\ref{item:second}) above shows $\dim M_i = \sum_{j \leq i} |B^j|$. Since $\bigsqcup_{j \leq i} B^j \subset M_i$ is a linearly independent set which has cardinality $\dim M_i$, it must be a basis of $M_i$. The linear independence of $B$ now implies that $B \cap M_i$ is precisely $\bigsqcup_{j \leq i} B^j \subset M_i$. 
\end{proof}

\subsection{Proof of Theorem~\ref{thm:main-thm-intro}}

We now prove Theorem~\ref{thm:main-thm-intro}. 
Let $\bas$ denote the set comprising the vectors in \eqref{eq:ffbasis-twisted}. By proposition~\ref{prop:zverma}, $\bas$ is a basis of $Z$. Let $\bas_n = \bas \cap Z_n$. Since equation~\eqref{eq:modejumps} implies that $\omd[p]_{-k}(Z_m) \subset Z_{m+k}$ for $1 \leq p \leq \ell$ and all $k,m \geq 0$, the set $\bas_n$ consists of vectors of the form \eqref{eq:ffbasis-twisted} with $\sum_i k_i = -n$.
In particular, this implies that  $|\bas_n| = p_\ell(n)$, the number of partitions of $n$ into parts of $\ell$ colours. Now, it is well-known that this number is also the dimension of $Z_n$ \cite[Remark 12.13]{kac}, i.e., $|\bas_n| = \dim Z_n = p_\ell(n)$.
So $\bas_n$ forms a basis of $Z_n$ for each $n$.

We now claim that $\bas_n$ is compatible with the Brylinski filtration $\{F^iZ_n\}$ on $Z_n$. Consider a vector $v \in \bas_n$. It has the form
\[ v=\omd[p_1]_{k_1}(\sigma)\; \omd[p_2]_{k_2}(\sigma)\; \cdots\; \omd[p_r]_{k_r}(\sigma)\, \hwvec\]
satisfying the conditions in \eqref{eq:ffbasis-twisted} and with $\sum_{i=1}^r k_i = -n$. Since $\hwvec$ is in $F^0 \vacm$, it follows from  Lemma~\ref{lem:omega-modes-property} that $v \in F^d \vacm \cap Z_n = F^dZ_n$, where $d = \sum_{i=1}^r d_{p_i}$. For each $d \geq 0$, define $\bas_n^d \subset F^dZ_n$ to be the set of vectors in \eqref{eq:ffbasis-twisted} satisfying $\sum_{i=1}^r d_{p_i}=d$ and $\sum_{i=1}^r k_i = -n$. It is clear from this definition  that $|\bas_n^d|$ is the coefficient of $t^dq^n$ in the product :
\[ \prod_{k=1}^{\ell}{\prod_{j=1}^{\infty}{(1-t^{d_{k}}q^j)^{-1}}}.                                                                                               \]
From \eqref{eq:hilbZ}, this is the same as $\dim \left( F^d Z_n \,/ F^{d-1}  Z_n\right)$. An appeal to Lemma~\ref{lem:filtration-grbasis} (with $M=Z_n$, $B = \bas_n$, $B^i = \bas_n^i$) finishes the proof. Corollary~\ref{cor:bryl-subspaces} follows easily. \qed


\begin{thebibliography}{10}

\bibitem{arakawa}
Tomoyuki Arakawa.
\newblock Representation theory of {$\scr W$}-algebras.
\newblock {\em Invent. Math.}, 169(2):219--320, 2007.

\bibitem{bakalov-kac}
Bojko Bakalov and Victor~G. Kac.
\newblock Twisted modules over lattice vertex algebras.
\newblock In {\em Lie Theory and Its Applications in Physics V}, pages 3--26.
  World Sci. Publ., River Edge, NJ, 2004.

\bibitem{bakalov-milanov}
Bojko Bakalov and Todor Milanov.
\newblock {$\scr W$}-constraints for the total descendant potential of a simple
  singularity.
\newblock {\em Compos. Math.}, 149(5):840--888, 2013.

\bibitem{BGM-physics}
Mikhail Bershtein, Pavlo Gavrylenko, and Andrei Marshakov.
\newblock Twist-field representations of $\walg$-algebras, exact conformal
  blocks and character identities.
\newblock {\em J. High Energ. Phys.}, 108, 2018.

\bibitem{BF}
Alexander Braverman and Michael Finkelberg.
\newblock Pursuing the double affine {G}rassmannian {I}: {T}ransversal slices
  via instantons on ${A}_k$-singularities.
\newblock {\em Duke Math. J.}, 152(2):175--206, 04 2010.

\bibitem{broer}
Bram Broer.
\newblock Line bundles on the cotangent bundle of the flag variety.
\newblock {\em Invent. Math.}, 113(1):1--20, 1993.

\bibitem{rkbryl}
Ranee-Kathryn Brylinski.
\newblock Limits of weight spaces, {L}usztig's $q$-analogs, and fiberings of
  adjoint orbits.
\newblock {\em Journal of the American Mathematical Society}, 2(3):517--533,
  1989.

\bibitem{dmmc}
Ivan Cherednik.
\newblock Difference {M}acdonald-{M}ehta conjecture.
\newblock {\em Internat. Math. Res. Notices}, (10):449--467, 1997.

\bibitem{dong-mason}
Chongying Dong and Geoffrey Mason.
\newblock Nonabelian orbifolds and the {B}oson-{F}ermion correspondence.
\newblock {\em Comm. Math. Phys.}, 163(3):523--559, 1994.

\bibitem{dong-nagatomo}
Chongying Dong and Kiyokazu Nagatomo.
\newblock Automorphism groups and twisted modules for lattice vertex operator
  algebras.
\newblock In {\em Recent developments in quantum affine algebras and related
  topics ({R}aleigh, {NC}, 1998)}, volume 248 of {\em Contemp. Math.}, pages
  117--133. Amer. Math. Soc., Providence, RI, 1999.

\bibitem{feigin-frenkel}
Boris Feigin and Edward Frenkel.
\newblock Integrals of motion and quantum groups.
\newblock In {\em Integrable systems and quantum groups ({M}ontecatini {T}erme,
  1993)}, volume 1620 of {\em Lecture Notes in Math.}, pages 349--418.
  Springer, Berlin, 1996.

\bibitem{fkrw}
Edward Frenkel, Victor Kac, Andrey Radul, and Weiqiang Wang.
\newblock {$\mathcal{W}_{1+\infty}$} and {$\mathcal{W}(\mathfrak{gl}_N)$} with
  central charge {$N$}.
\newblock {\em Comm. Math. Phys.}, 170(2):337--357, 1995.

\bibitem{frenkelkacwaki}
Edward Frenkel, Victor Kac, and Minoru Wakimoto.
\newblock Characters and fusion rules for {$W$}-algebras via quantized
  {D}rinfeld-{S}okolov reduction.
\newblock {\em Comm. Math. Phys.}, 147(2):295--328, 1992.

\bibitem{FLMbook}
Igor Frenkel, James Lepowsky, and Arne Meurman.
\newblock {\em Vertex operator algebras and the {M}onster}, volume 134 of {\em
  Pure and Applied Mathematics}.
\newblock Academic Press, Inc., Boston, MA, 1988.

\bibitem{frenkel-huang-lepowsky}
Igor~B. Frenkel, Yi-Zhi Huang, and James Lepowsky.
\newblock On axiomatic approaches to vertex operator algebras and modules.
\newblock {\em Mem. Amer. Math. Soc.}, 104(494):viii+64, 1993.

\bibitem{frenkel-kac}
Igor~B. Frenkel and Victor~G. Kac.
\newblock Basic representations of affine {L}ie algebras and dual resonance
  models.
\newblock {\em Invent. Math.}, 62(1):23--66, 1980/81.

\bibitem{rkg}
Ranee~K. Gupta.
\newblock Characters and the {$q$}-analog of weight multiplicity.
\newblock {\em J. London Math. Soc. (2)}, 36(1):68--76, 1987.

\bibitem{JLZ}
Anthony Joseph, Gail Letzter, and Shmuel Zelikson.
\newblock On the {B}rylinski-{K}ostant filtration.
\newblock {\em J. Amer. Math. Soc.}, 13(4):945--970, 2000.

\bibitem{kac-infdim-dedekind}
Victor~G. Kac.
\newblock Infinite-dimensional algebras, {D}edekind's {$\eta $}-function,
  classical {M}\"{o}bius function and the very strange formula.
\newblock {\em Adv. in Math.}, 30(2):85--136, 1978.

\bibitem{kac}
Victor~G. Kac.
\newblock {\em Infinite dimensional {L}ie algebras}.
\newblock Cambridge University Press, third edition, 1990.

\bibitem{kac-vafb}
Victor~G. Kac.
\newblock {\em Vertex algebras for beginners}, volume~10 of {\em University
  Lecture Series}.
\newblock American Mathematical Society, Providence, RI, second edition, 1998.

\bibitem{Kac-Kazhdan-Lepowsky-Wilson}
Victor~G. Kac, David~A. Kazhdan, James Lepowsky, and Robert~L. Wilson.
\newblock Realization of the basic representations of the {E}uclidean {L}ie
  algebras.
\newblock {\em Adv. in Math.}, 42(1):83--112, 1981.

\bibitem{kac-peterson-112constructions}
Victor~G. Kac and Dale~H. Peterson.
\newblock 112 constructions of the basic representation of the loop group of
  ${E}_8$.
\newblock In {\em Symposium on Anomalies, Geometry, Topology (Chicago 1985),
  276-298}. World Sci. Publ., Singapore, 1985.

\bibitem{kac-raina}
Victor~G. Kac and Ashok~K. Raina.
\newblock {\em Bombay lectures on highest weight representations of
  infinite-dimensional {L}ie algebras}, volume~2 of {\em Advanced Series in
  Mathematical Physics}.
\newblock World Scientific Publishing Co., Inc., Teaneck, NJ, 1987.

\bibitem{kato}
Shin{-}ichi Kato.
\newblock Spherical functions and a {$q$}-analogue of {K}ostant's weight
  multiplicity formula.
\newblock {\em Invent. Math.}, 66(3):461--468, 1982.

\bibitem{lepowsky-calctwistvo}
James Lepowsky.
\newblock Calculus of twisted vertex operators.
\newblock {\em Proc. Nat. Acad. Sci. U.S.A.}, 82(24):8295--8299, 1985.

\bibitem{lusztig}
George Lusztig.
\newblock Singularities, character formulas, and a {$q$}-analog of weight
  multiplicities.
\newblock In {\em Analysis and topology on singular spaces, {II}, {III}
  ({L}uminy, 1981)}, volume 101 of {\em Ast\'{e}risque}, pages 208--229. Soc.
  Math. France, Paris, 1983.

\bibitem{slofstra}
William Slofstra.
\newblock A {B}rylinski filtration for affine {K}ac--{M}oody algebras.
\newblock {\em Advances in Mathematics}, 229(2):968 -- 983, 2012.

\bibitem{svis-kfp}
Sankaran Viswanath.
\newblock Kostka-{F}oulkes polynomials for symmetrizable {K}ac-{M}oody
  algebras.
\newblock {\em S\'em. Lothar. Combin.}, 58:Art. B58f, 2008.

\bibitem{zhu}
Yongchang Zhu.
\newblock Modular invariance of characters of vertex operator algebras.
\newblock {\em J. Amer. Math. Soc.}, 9(1):237--302, 1996.

\end{thebibliography}
\end{document}